\documentclass{amsart}

\usepackage{comment}

\usepackage[all]{xy}   

\usepackage{hyperref} 

\usepackage{amsmath,amssymb,amsthm, amsfonts, dsfont}
\usepackage[usenames,dvipsnames]{xcolor}
\usepackage{tikz, mathtools,listings}

\definecolor{mygreen}{rgb}{0,0.3,0}
\definecolor{mygray}{rgb}{0.5,0.5,0.5}
\definecolor{mymauve}{rgb}{0.58,0,0.82}

\theoremstyle{plain}
\newtheorem{lem}{Lemma}[section]
\newtheorem{cor}[lem]{Corollary}
\newtheorem{prop}[lem]{Proposition}
\newtheorem{thm}[lem]{Theorem}

\theoremstyle{definition}
\newtheorem{ex}[lem]{Example}
\newtheorem{rem}[lem]{Remark}
\newtheorem{dfn}[lem]{Definition}

\newcommand{\edge}[1]{\ar@{-}[#1]}      
\newcommand{\lulab}[1]{\ar@{}[l]_<<{{}_{#1}}}   
\newcommand{\node}{*=0{\bullet}}                       

\newcommand{\zz}{\mathbb{Z}}    
\newcommand{\qq}{\mathbb{Q}}    
\newcommand{\rr}{\mathbb{R}}     
\newcommand{\cc}{\mathbb{C}}   

\newcommand{\HH}{\mathcal{H}}   

\newcommand{\RS}{\Phi}    
\newcommand{\SR}{\Delta}   

\newcommand{\cl}{\Lambda}  
\newcommand{\rl}{\cl_r}  
\newcommand{\wl}{\cl_w}  

\newcommand{\al}{\alpha}   
\newcommand{\be}{\beta}    
\newcommand{\de}{\delta}   
\newcommand{\la}{\lambda} 

\newcommand{\hra}{\hookrightarrow} 

\newcommand{\tA}{\mathrm{A}}  
\newcommand{\tB}{\mathrm{B}}  
\newcommand{\tC}{\mathrm{C}}  
\newcommand{\tD}{\mathrm{D}}  
\newcommand{\tE}{\mathrm{E}}  
\newcommand{\tF}{\mathrm{F}}  
\newcommand{\tG}{\mathrm{G}}  
\newcommand{\tH}{\mathrm{H}} 
\newcommand{\tI}{\mathrm{I}}    

\DeclareMathOperator{\Hom}{\mathrm{Hom}}    
\newcommand{\om}{\omega}  
\newcommand{\DD}{\mathcal{D}}    
\newcommand{\id}{\mathrm{id}}   
\newcommand{\OF}{\zz_{\RS}} 

\DeclareMathOperator{\GL}{\mathrm{GL}}  
\DeclareMathOperator{\SL}{\mathrm{SL}}
\DeclareMathOperator{\PGL}{\mathrm{PGL}}

\newcommand{\NN}{\mathcal{N}} 
\DeclareMathOperator{\Sym}{\mathrm{Sym}}    
\newcommand{\NH}{\mathfrak{N}} 
\newcommand{\xra}[1]{\xrightarrow{#1}} 
\DeclareMathOperator{\End}{\mathrm{End}}    

\newcommand{\DI}{\mathit{\Delta}} 
\newcommand{\eps}{\epsilon} 
\newcommand{\veps}{\varepsilon} 
\newcommand{\II}{\mathcal{I}} 
\newcommand{\ii}{\mathbf{i}} 
\newcommand{\jj}{\mathbf{j}}
\newcommand{\kkk}{\mathbf{k}}

\newcommand{\kk}{\mathds{k}} 
\DeclareMathOperator{\SC}{\mathrm{H}} 
\DeclareMathOperator{\CH}{\mathrm{CH}} 
\DeclareMathOperator{\image}{\mathrm{Im}}

\DeclareMathOperator{\K}{\mathrm{K}_0} 
\DeclareMathOperator{\CK}{\mathrm{CK}} 
\DeclareMathOperator{\hh}{\mathcal{A}}  
\DeclareMathOperator{\AC}{\mathrm{\Omega}} 
\newcommand{\LL}{\mathbb{L}}      
\newcommand{\Ll}{\mathcal{L}}   

\DeclareMathOperator{\MM}{\mathcal{M}} 
\newcommand{\OO}{\mathcal{O}} 
\newcommand{\SP}{\mathcal{S}} 
\newcommand{\DF}{\mathfrak{D}}

\begin{document}

\title[Nil-Hecke rings and the Schubert calculus]{Nil-Hecke rings and the Schubert calculus}

\author[E.~Richmond, K.~Zainoulline]{Edward Richmond, Kirill Zainoulline}

\address[Edward Richmond]{Department of Mathematics, Oklahoma State University, Stillwater, OK 74078, USA}
\email{edward.richmond@okstate.edu}
\urladdr{https://math.okstate.edu/people/erichmond/}

\address[Kirill Zainoulline]{Department of Mathematics and Statistics, University of Ottawa, 150 Louis-Pasteur, Ottawa, ON, K1N 6N5, Canada}
\email{kirill@uottawa.ca}
\urladdr{https://mysite.science.uottawa.ca/kzaynull/}

\subjclass[2010]{}

\begin{abstract}
The purpose of the present notes is to give a self-contained exposition
on the use of the techniques of Nil-Hecke algebras in the localization approach to the equivariant Schubert calculus
for cohomology of flag varieties. We also demonstrate how this techniques can be applied to non-crystallographic root systems
as well as to study (connective) $K$-theory of flag varieties.
\end{abstract}

\maketitle


\tableofcontents

\section{Introduction}

Hecke-type algebras and the associated localization techniques play an essential role in the modern equivariant Schubert calculus. It was first observed by Demazure in~\cite{Dem73} and then by Konstant and Kumar in \cite{KK86, KK90} that the algebras of divided-difference operators (nil-Hecke and $0$-Hecke algebras) can be efficiently used to describe and study (equivariant) cohomology and $K$-theory rings of flag varieties and Kac-Moody groups. Later this approach was applied to many other generalized cohomology theories (e.g., to algebraic and complex oriented cobordism in \cite{BE87,BE90} and \cite{HMSZ,CPZ,CZZ,CZZ1,CZZ2}). Indeed, almost all geometric structures which come with such generalized cohomology can be reconstructed in a purely algebraic way using the respective Hecke-type algebras. For example, the intersection product on cohomology can be defined as the dualized co-product of the respective nil-Hecke ring. Hence, all the properties of the intersection product can be deduced from those of the coproduct.

In these notes, we give a self-contained and detailed exposition of the interplay between nil-Hecke algebras, the coproduct structure and the cohomology of flag varieties.  In particular, we showcase how nil-Hecke algebras can be used a practical tool to make calculations in cohomology.  We also demonstrate how the same techniques can be applied to non-crystallographic root systems as well as to study (connective) $K$-theory of flag varieties.  Many of the results stated in these notes readily extend to root systems and flag varieties of Kac-Moody type, however for the sake of exposition and simplicity, we restrict our focus to finite root systems (although not necessarily crystallographic) and finite dimensional flag varieties.  For more on Kac-Moody nil-Hecke algebras and flag varieties, see \cite{Ku02}.

Our paper is organized as follows. In section~\ref{Section_typeA_example} we start by discussing root systems, root and weight lattices, Weyl and Coxeter groups.
These objects serve as the algebraic and combinatorial foundation for constructing and analyzing nil-Hecke algebras. In the next section~\ref{sec:nilHecke} we define the nil-Hecke ring associated to a root datum $\RS\hra \cl^\vee$ and give several examples. Observe that our definition applies to all root systems and lattices (not necessarily crystallographic).
We then realize the nil-Hecke ring as a sub-ring of a twisted group algebra and as a subring of certain operators on the polynomial ring of the root lattice. These two presentations provide important computational tools for dealing with nil-Hecke rings. Then in section~\ref{sec:coprod} we equip the localized twisted group algebra $Q_W$ with a co-product structure that plays an important role in the Schubert calculus of flag varieties. This coproduct naturally restricts to nil-Hecke subrings and also leads to the augmented coproduct (corresponding to the usual cohomology). We discuss several recursive formulas for its structure coefficients and provide various examples of computations. In section~\ref{sec:twistedgr} we study duals of the twisted group algebra, nil-Hecke ring and its relations to Schubert calculus.  In section~\ref{sec:eqcoh} we outline the classical connections between torus-equivariant singular cohomology of flag varieties and nil-Hecke rings. In section~\ref{sec:eqkth} we demonstrate how the results of previous sections can be extended to the context of the so-called connective $K$-theory.  Finally in the last section~\ref{sec:sagemath} we present some SageMath (Python) code for computing Schubert calculus structure coefficients for flag varieties of finite type.  We remark that many of the results that appear in Section 4-6 on Nil-Hecke rings follow \cite[Chapters 11-12]{Ku02}.

\

\paragraph{\it Acknowledgements}  E.R. was supported by a grant from the Simons Foundation 941273.  K.Z. was partially supported by the 
NSERC Discovery grant RGPIN-2022-03060, Canada.


\section{Root systems and Coxeter groups}\label{Section_typeA_example}

In this section, we discuss root systems, root and weight lattices, Weyl and Coxeter groups.
These objects serve as the algebraic and combinatorial foundation for constructing and analyzing nil-Hecke algebras.
There are many references on these subjects. In our presentation we primarily follow \cite{Humphreys} and \cite{Bou}.

We start from the following classical motivating example.

Let $n\geq 1$ and consider the Euclidean space $\rr^{n+1}$ with standard basis $\{e_0,\ldots,e_n\}$.
Let $\HH \subset \rr^{n+1}$ be the hyperplane consisting of vectors whose sum of coordinates is zero
\[
\HH:=\Big\{\sum_{i=0}^n a_i e_i \colon \sum_{i=0}^n a_i=0\Big\}.
\]
Let $\RS$ denote the subset of $\HH$ consisting of all differences of standard vectors
\[
\RS:=\{e_i-e_j\}_{i\neq j}.
\]
Clearly $\RS\cap m\RS\neq \emptyset$ if and only if $m=\pm 1$.
Observe that each vector from $\RS$ can be uniquely expressed as a $\zz$-linear combination of vectors from the subset
\[
\SR:=\{e_{i-1}-e_{i}\}_{i=1,\ldots,n}.
\]
Moreover, if $\al\in \RS$ and
\[
\al=\sum_{i=1}^n c_i (e_{i-1}-e_{i}),
\]
then either all coefficient $c_i\geq 0$, or all $c_i\leq 0$.
In other words, $\RS$ splits into two disjoint subsets $\RS^+$ and $\RS^-$, where
\[
\RS^+:=\{e_i-e_j\}_{i<j}\quad \text{and}\quad \RS^-:=\{e_i-e_j\}_{i>j}.
\]
Let $\cl$ denote the $\zz$-linear span of $\SR$.
Observe that $\cl$ is a free $\zz$-module of rank $n$ and $\cl\otimes_\zz \rr \simeq \HH$.

Given $\al\in \RS$, consider the orthogonal reflection $s_{\al}$ which fixes the hyperplane orthogonal to $\al=e_i-e_j$.
This reflection is given by the following formula
\[
s_\al(\be):=\be- (\al,\be)\, \al,
\]
where $(\cdot,\cdot)$ is the usual dot-product on $\rr^{n+1}$ given by $(e_i,e_j):=\de_{i,j}$.
Observe that $s_\al$ simply switches $e_i$ and $e_j$ (the $i$th and the $j$th coordinates)
and, hence, it leaves both $\RS$ and $\cl$ invariant:
\[
s_\al(\be)\in \RS\; (\text{resp.}\,\cl),\; \text{for all}\; \be \in \RS\; (\text{resp.}\, \cl).
\]

Let $W$ be the group generated by all reflections $s_\al$, $\al\in \SR$,
where the multiplication is given by the composition.
Since it acts by permutations of the set of standard vectors $\{e_0,\ldots,e_n\}$,
it can be identified with the symmetric group $S_{n+1}$ on the set of indices $\{0,\ldots, n\}$.
By the very definition it leaves $\RS$ and $\cl$ invariant.

\medskip

The introduced subset of vectors $\RS$ provides an example of a {\em root system}
of type $\tA$ and of rank $n$ denoted $\tA_n$:

\begin{itemize}
\item
elements of $\RS$ are called roots,
\item
elements of $\SR$ are called simple roots,
\item
elements of $\RS^+$ and $\RS^-$ are called positive and negative roots, respectively,
\item
the free $\zz$-module $\cl$ is called the root lattice,
\item
the group $W$ is called the Weyl group,
\item
elements $s_\al\in W$, $\al\in \SR$ are called simple reflections.
\end{itemize}


\subsection{\it Root systems and lattices.}\label{rddef}
We now provide definitions of a general root system and of a root datum.
These generalize the previous example and serve as algebraic and combinatorial foundations for most of the objects discussed later.

\begin{dfn}
Let $\RS$ be a finite (non-empty) subset of $\rr^N\setminus\{0\}$
(here we view $\rr^N$ as a Euclidean space as before).
For each vector $\al\in\RS$ we define a linear transformation $s_\al\colon \rr^N\to \rr^N$ by
\[
s_\al(x):=x-\tfrac{2(\al,x)}{(\al,\al)}\, \al,\quad x\in \rr^N.
\]
We say $\RS$ is a {\em root system} if the following conditions are satisfied:
\begin{enumerate}
\item[(i)]
If $m\in\rr$ and $\RS\cap m\RS\neq\emptyset$, then $m=\pm 1$, and
\item[(ii)]
$s_\al(\RS)\subseteq \RS$ for all $\al \in \RS$.
\end{enumerate}
\end{dfn}

Next, we define a root datum which pairs a root system with a lattice.
The definition follows the one given in \cite[Exp.~XXI, \S1.1]{SGA}.

\begin{dfn}\label{Def:root_datum_int}
A {\em root datum} is a finitely generated free $\zz$-module $\cl$ with a nonempty finite subset $\RS$ together with a set inclusion
\[
\RS \hra \cl^\vee, \quad \al \mapsto \al^\vee
\]
into the dual $\cl^\vee:=\Hom_\zz(\cl,\zz)$ such that
\begin{enumerate}
\item[(i)]
If $m\in \zz$ and $\RS \cap m\RS\neq \emptyset$, then $m=\pm 1$,
\item[(ii)]
$\al^\vee(\al)=2$ for all $\al \in \RS$, and
\item[(iii)]
for all  $\al,\be \in \RS$, we have
\[
\be - \al^\vee(\be)\,\al \in \RS
\quad\text{and}\quad
\be^\vee - \be^\vee(\al)\,\al^\vee \in \RS^\vee.
\]
\end{enumerate}
Here $\RS^\vee$ denotes the image of $\RS$ in $\cl^\vee$,
and the elements of $\RS$ (resp. $\RS^\vee$) are called roots (resp. coroots).
\end{dfn}

The $\zz$-submodule generated by $\RS$ in $\cl$ is called the {\em root lattice} and is denoted by $\rl$.
A root datum is called {\em semisimple} if
\[
\cl\otimes_\zz \qq=\rl \otimes_\zz \qq.
\]
From now on, by a root datum we will always mean a semisimple one.
The $\zz$-submodule of $\cl_\qq:=\cl\otimes_\zz \qq$ generated by all $x \in \cl_\qq$
such that $\al^\vee(x)\in \zz$ for all $\al\in\RS$ is called the {\em weight lattice} and is denoted by $\wl$.

By definition we have
\[
\rl\subseteq \cl \subseteq \wl\quad\text{ and }\quad \rl\otimes_\zz\qq=\cl_\qq=\wl\otimes_\zz \qq.
\]
The $\qq$-rank of $\cl_\qq$ is called the {\em rank} of the root datum.
Each root datum $\RS\hra \cl^\vee$ gives rise to a root system
by considering the real vector space $\cl_\rr:=\cl \otimes_\zz \rr$
equipped with a compatible Euclidean product $(\cdot,\cdot)$ such that
\[
\al^\vee(x)=\tfrac{2(\al,x)}{(\al,\al)}.
\]
While a root datum determines a root system, the converse is not true.
For example, taking $\cl=\cl_r$ compared to $\cl=\cl_w$ can produce non-isomorphic root datum.

It can be shown that the root lattice $\rl$ admits a $\zz$-basis $\SR=\{\al_1,\ldots,\al_n\}$
such that each $\al \in \RS$ is a linear combination of $\al_i$'s with
either all positive or all negative coefficients and $n$ is the rank of the root datum.
Hence the set $\RS$ splits into two disjoint subsets
\[
\RS = \RS^+ \sqcup \RS^-,
\]
where $\RS^+$ (resp. $\RS^-$) is called the set of positive (resp. negative) roots.
The $\zz$-basis $\SR$ is called the set of {\em simple roots},
We define the set of {\em fundamental weights} $\{\om_1,\ldots,\om_n\}\subset \wl$
to be Kronecker dual to the simple coroots $\{\al_1^\vee,\ldots,\al_n^\vee\}$,
i.e.,
\[\al_i^\vee(\om_j):=\de_{i,j}.\]
The fundamental weights form a $\zz$-basis of the weight lattice $\wl$.
The {\em Cartan matrix} $C:=[c_{ij}]_{1\leq i,j\leq n}$ of the root datum is defined by the values
\[
c_{ij}:=\al_j^\vee(\al_i).
\]
By definition we have $\al_i=\sum_{j=1}^n c_{ij}\,\om_j$.
In other words, the Cartan matrix expresses simple roots in terms of fundamental weights.
Observe that the determinant of the Cartan matrix coincides
with the number of elements in the quotient group $\cl_w/\cl_r$.

\begin{ex}
Consider the root system of type $\tA_2$.
Here we consider the hyperplane in $\rr^3$ given by
\[
\HH=\{a_1e_1+a_2e_2+a_3e_3\ |\ a_1+a_2+a_3=0\}
\]
and the subset of vectors
\[
\RS=\{\pm (e_1-e_2), \pm (e_2-e_3),\pm (e_1-e_3)\}\subset \HH.
\]
The subset $\RS$ forms a root system of rank 2 where
\[
\al_1:=e_1-e_2\quad\text{and}\quad\al_2:=e_2-e_3
\]
is the standard choice of simple roots.
The Cartan matrix of type $\tA_2$ is
\[
C=\left[\begin{matrix} 2 & -1 \\ -1 & 2 \end{matrix}\right].
\]
We can diagrammatically represent this root system as follows:
\[
\begin{tikzpicture}
\node[inner sep=0pt] at (0,0)
    {\includegraphics[scale=0.5]{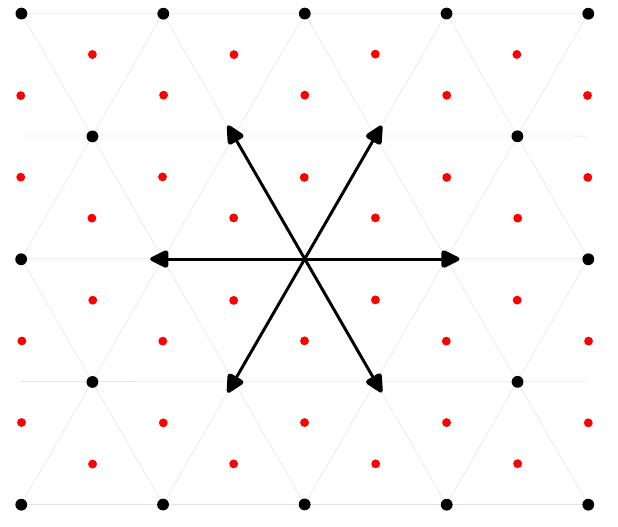}};
\node at (1.6,0) {$\al_1$};
\node at (-1.1,1.1) {$\al_2$};
\node at (1,1.4) {$\al_1+\al_2$};
\end{tikzpicture}
\]
Here the black vertices represent the root lattice $\cl_r$ and the black and red dots together represent the weight lattice $\cl_w$.
In particular, $\cl_r$ is the $\zz$-span of $\al_1$ and $\al_2$ and $\cl_w$ is the $\zz$-span of fundamental weights
\[
\om_1:=\tfrac{1}{3}(2\al_1+\al_2)\quad\text{and}\quad\om_2:=\tfrac{1}{3}(\al_1+2\al_2).
\]
Note that calculating the fundamental weights in standard coordinates gives
$\om_1=e_1$ and $\om_2=e_1+e_2$.
The lattices $\cl_r$ and $\cl_w$ are two examples of a root datum of type $\tA_2$.
The quotient group $\cl_w/\cl_r$ is cyclic of order $3$.
\end{ex}


\subsection{\it Dynkin diagrams and classification.}
A root datum is called {\em irreducible} if it cannot be represented as a direct sum of root data,
i.e., $\cl$ cannot be written as $\cl=\cl_1\oplus \cl_2$,
where $\RS_1\subset \cl_1$ and $\RS_2\subset \cl_2$ are the root data.

To any irreducible root datum we associate a multi-graph $\DD$ called the {\em Dynkin diagram}.
Its vertices correspond to the set of simple roots $\SR=\{\al_1,\ldots,\al_n\}$, and
the number of edges connecting two different vertices $\al_i$ and $\al_j$ is determined
by the product $c_{ij}\cdot c_{ji}$ of coefficients of the Cartan matrix.
If there is more than one edge connecting $\al_i$ and $\al_j$,
we assign a direction to the edges between $\al_i$ and $\al_j$.
Specifically if $c_{ij}<c_{ji}$, then $\al_j$ points to $\al_i$.
In this case, we also have that $(\al_i,\al_i)<(\al_j,\al_j)$.
If there is a single edge, between $\al_i$ and $\al_j$, then $(\al_i,\al_i)=(\al_j,\al_j)$.

All Dynkin diagrams of irreducible root data are classified and consist of the following types.
Here we label a vertex $i$ when it is associated to the simple root $\al_i$
(our enumeration of vertices follows Bourbaki; the lower index is the rank $n$):

\medskip

\noindent
\begin{tabular}{ll}
Classical types: & \\
$\xymatrix@C=2.5em{
\tA_n: & \node \lulab{1} \edge{r} & \node \lulab{2} \edge{r} &
\node \lulab{3} \ar@{..}[r] & \node \lulab{n-1} \edge{r} & \node \lulab{n}
}$
& \qquad
$\xymatrix@C=2.5em{
\tC_n:  & \node \lulab{1} \edge{r} & \node \lulab{2} \ar@{..}[r]
& \node \lulab{n-2} \edge{r} & \node \lulab{n-1} \ar@{=}[r] |-{<}\ar@{=}[r]
& \node \lulab{n}
}$
\\
$\xymatrix@C=2.5em{
\tB_n:& \node \lulab{1} \edge{r} & \node \lulab{2} \ar@{..}[r]
& \node \lulab{n-2} \edge{r} & \node \lulab{n-1}  \ar@{=}[r] |-{>} \ar@{=}[r]
& \node \lulab{n}
}$
& \qquad
$\xymatrix@C=2em@R=1ex{
\tD_n: & \node \lulab{1} \edge{r} & \node \lulab{2} \ar@{..}[r]
& \node \lulab{n-3} \edge{r} & \node \lulab{n-2}\edge{r} & \node \lulab{n-1}\\
&&&& \node \lulab{n}\edge{u} &
}$
\end{tabular}

\medskip

\noindent
\begin{tabular}{ll}
Exceptional types: & \\
$\xymatrix@=2em{
\tG_2: &\node \lulab{1}  \ar@3{-}[r]|-{<}\ar@3{-}[r]
& \node \lulab{2}
}$
&
$\xymatrix@C=1.5em@R=1ex{
\tE_7: &\node\lulab{1}\edge{r}& \node \lulab{3} \edge{r}& \node \lulab{4} \edge{r}\edge{d} &
\node \lulab{5} \edge{r} & \node \lulab{6}
\edge{r} & \node \lulab{7}\\
&&& \node \lulab{2} &&&
}$
\\
$\xymatrix@=2em{
\tF_4: & \node \lulab{1} \edge{r} & \node \lulab{2}  \ar@{=}[r]|-{>}\ar@{=}[r]
& \node \lulab{3}\edge{r} & \node \lulab{4}
}$
&
$\xymatrix@C=1.5em@R=1ex{
\tE_8: &\node\lulab{1}\edge{r}&\node\lulab{3}\edge{r}& \node \lulab{4} \edge{r}\edge{d}& \node \lulab{5} \edge{r} &
\node \lulab{6} \edge{r} & \node \lulab{7}
\edge{r} & \node \lulab{8}\\
&&& \node \lulab{2} &&&&
}$
\\
$\tE_6:
\xymatrix@=1.5em@R=1.5ex{
& \node \lulab{1} \edge{r}& \node \lulab{3} \edge{r} &
\node \lulab{4} \edge{r}\edge{d} & \node \lulab{5}
\edge{r} & \node \lulab{6}\\
&&& \node \lulab{2} &&
}$
&
\end{tabular}

\medskip

It can be shown that an irreducible root datum is determined uniquely by

\begin{itemize}
\item its Dynkin diagram $\DD$ and
\item the intermediate lattice $\cl_r\subseteq
\cl\subseteq \cl_w$.
\end{itemize}

If $\cl=\wl$ (resp. $\cl=\rl$), then the root datum is called simply connected (resp. adjoint),
and it will be denoted by $\DD^{sc}_n$ (resp. $\DD_n^{ad}$),
where $\DD={\rm A,B,C,D,E,F,G}$ is one of the Dynkin diagrams and $n$ is its rank.


\subsection{\it Coxeter groups and Weyl groups}
For each root $\al \in \RS$, define a $\zz$-linear map $s_\al\colon \cl_w \to \cl_w$, by
\[
s_\al(\la):=\la - \al^\vee (\la) \al, \quad \la\in \cl_w.
\]
The linear map $s_\al$ is called the {\em reflection} with respect to the root $\al$.
By definition we have $s_\al^2=\id$.
The group $W$ generated by all reflections $s_\al$, $\al\in \RS$ is called the {\em Weyl group} of the root datum.
It can be shown that the Weyl group has a presentation
\[
W=\langle s_1,\ldots,s_n \colon (s_is_j)^{m_{ij}}=1 \rangle
\]
where $s_i=s_{\al_i}$ denotes the generator corresponding to the simple root $\al_i\in \SR$.
So $W$ is the finite Coxeter group.
The {\em Coxeter exponents} $m_{ij}$ are determined by the coefficients of the Cartan matrix:

If $i=j$, then  $m_{ii}:=1$.  Otherwise, if $i\neq j$, then:
\[
\begin{aligned}
m_{ij} &:=2 \text{ if }c_{ij}c_{ji}=0 \text{ (no edge), } \\
m_{ij} &:=3 \text{ if }c_{ij}c_{ji}=1 \text{ (single edge), }\\
m_{ij} &:=4 \text{ if }c_{ij}c_{ji}=2 \text{ (double edge), }\\
m_{ij} &:=6 \text{ if }c_{ij}c_{ji}=3 \text{ (triple edge). }
\end{aligned}
\]
(Since $\RS$ is finite, these are all possible values of Coxeter exponents.)

By definition each element of $W$ can be written as a product of simple reflections
\[
w=s_{i_1}s_{i_2}\cdots s_{i_r}.
\]
The minimal number of generators required to express an element $w\in W$ is unique
and is called the {\em length of $w$}.
We denote the length by $\ell(w)$.
Any expression of $w$ as a product of $\ell(w)$ simple reflections is called {\em reduced} expression (or {\em reduced word}).


\subsection{\it Finite real root systems.}
We now show how root datum can be generalized to finite real root systems.

Let $\RS\subseteq \rr^N$ be a root system and for each $\al\in \RS$. We define $\al^\vee\in (\rr^N)^*$ by
\begin{equation}\tag{*}
\al^\vee(x) := \tfrac{2(\al,x)}{(\al,\al)}, \quad  x\in \rr^N.
\end{equation}
If $\al^\vee(\be)\in \zz$ for all $\al,\be\in \RS$, then $\al^\vee$ defines an element in the dual $\zz$-module $\cl_r^\vee$,
and the map $\al\mapsto\al^\vee$ gives rise to an inclusion $\RS \hra \cl_r^\vee$,
that satisfies the properties of a root datum.
This gives rise the following definition in \cite[\S2.9]{Humphreys}

\begin{dfn}
We say a finite root system $\RS\subseteq \rr^N\setminus\{0\}$ is {\em crystallographic}
if the values $\tfrac{2(\al,\be)}{(\al,\al)}\in\zz$ for all $\al,\be\in \RS$.
\end{dfn}

If a root system is not crystallographic (i.e. $\al^\vee(\be)\notin \zz$ for some $\al,\be\in\RS$),
then assigning coroots using the formula~(*) does not produce a well-defined root datum.
However, with an appropriate change in coefficient ring, a root datum can be formally defined for any root system (see below).

Let $\RS\subseteq \rr^N\setminus\{0\}$ be a root system and
define $\OF$ to be the subring of $\rr$ generated by the values $\al^\vee(\be)$ for all $\al,\be\in \RS$.
We call $\OF$ the {\em coefficient ring of the root system} $\RS$ and
it can be viewed as the smallest subring of $\rr$ containing the values $\al^\vee(\be)$.
If $\RS$ is crystallographic, then $\OF=\zz$.
The following is a generalization of Definition~\ref{Def:root_datum_int} that allows for non-crystallographic root datum.

\begin{dfn}\label{Def:root_datum_real}
Let $\RS\subseteq \rr^N\setminus\{0\}$ be a root system and
for each $\al\in\RS$, define the coroot $\al^\vee\in{(\rr^N)}^*$ as in~(*).
A {\em root datum with coefficients in $\OF$} is a free $\OF$-module $\cl$ containing $\RS$
together with a set inclusion
\[
\RS\hra \cl^\vee, \quad \al\mapsto \al^\vee
\]
in the dual $\cl^\vee:=\Hom_{\OF}(\cl,\OF)$ such that
\begin{enumerate}
\item[(i)]
If $m\in \OF$ and $\RS \cap m\RS\neq \emptyset$, then $m=\pm 1$,
\item[(ii)]
$\al^\vee(\al)=2$ for all $\al \in \RS$, and
\item[(iii)]
for any  $\al,\be \in \RS$, we have
\[
\be - \al^\vee(\be)\,\al \in \RS\quad\text{and}\quad  \be^\vee - \be^\vee(\al)\,\al^\vee \in \RS^\vee.
\]
\end{enumerate}
\end{dfn}

Define the \emph{root lattice} $\cl_r$ to be the $\OF$-linear span of roots in $\RS$.
Taking $\cl=\cl_r$ provides an example of a root datum with coefficients in $\OF$.
As with crystallographic root datum, $\cl_r$ admits an $\OF$-basis of simple roots
$\SR=\{\al_1,\cdots,\al_n\}\subseteq\RS$ and
the Weyl group $W=\langle s_1,\ldots,s_n\colon (s_is_j)^{m_{ij}}=1\rangle$.
Let $K$ denote the field of fractions of $\OF$.
If $\RS$ is crystallographic then $K=\qq$.
We say a root datum is {\em semisimple} if
\[
\cl \otimes_{\OF} K= \cl_r \otimes_{\OF} K.
\]

Define the \emph{weight lattice}
\[
\cl_w:=\{\la\in \cl_r\otimes_{\OF} K\colon \al^\vee(\la)\in {\OF} \text{ for all }\al\in \RS\}.
\]
We prove the following lemma:

\begin{lem}\label{lem:semisimple_lattices}
Let $\RS\subseteq \rr^N\setminus\{0\}$ be a root system and
let $\RS\hra \cl^\vee$ be a semisimple root datum with coefficients in $\OF$.
Then the weight lattice $\cl_w$ is a finitely generated free $\OF$-module and $\cl_r\subseteq\cl\subseteq \cl_w$.
\end{lem}

\begin{proof}
Since it is semisimple, it immediately follows that $\cl_r\subseteq\cl\subseteq \cl_w$.
Let $\SR:=\{\al_1,\ldots,\al_n\}$ denote the simple roots of $\RS$.
If $\al\in \RS$, then by \cite[\S1.5]{Humphreys}, there exists $w\in W$ such that $\al=w(\al_i)$ for some $\al_i\in \SR$.
By induction on the length of $w$, we have that $\al=\sum_{i=1}^n b_i\,\al_i$ for some coefficients $b_i\in \OF$.
Hence, if $\al_j^{\vee}(\la)\in \OF$ for all $\al_j\in\SR$,
then $\al^{\vee}(\la)\in \OF$.
Let $\la\in \cl_r\otimes_{\OF} K$ and write
\[
\la=\sum_{i=1}^n c_i\, \al_i, \quad c_i\in K.
\]
We have that $\la\in \cl_w$ if and only if
\[
\al_j^\vee(\la)=\sum_{i=1}^nc_i\, \al_j^\vee(\al_i)\in \OF\;\text{ for all }\al_j\in\SR.
\]
These conditions are equivalent to $C\cdot(c_1,\ldots,c_n)^\mathrm{t}\in (\OF)^n$,
where $C$ is the Cartan matrix of $\RS$.
Hence $\cl_w=C^{-1}(\cl_r)$ is a finitely generated, free $\OF$-module.
\end{proof}

As in the crystallographic case, we will only consider semisimple root datum.

\begin{ex}\label{Example:B2}
We contrast two examples of root system $\tB_2$.
First, we have the crystallographic root system of type $\tB_2$ given by the Cartan matrix
\[
C=\left[\begin{matrix} 2 & -2 \\ -1 & 2 \end{matrix}\right].
\]
The root system and root lattice can be visualized by the following diagram:
\[
\begin{tikzpicture}
\node[inner sep=0pt] at (0,0)
    {\includegraphics[scale=0.75]{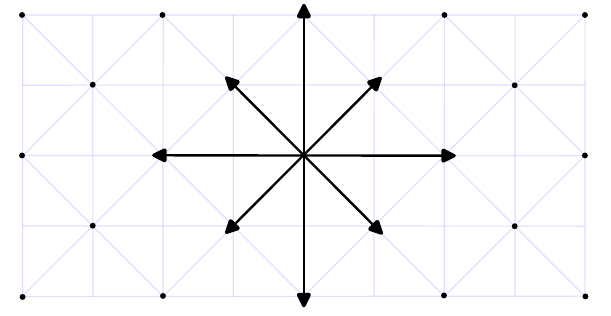}};
\node at (2.3,0) {$\al_1$};
\node at (-1.3,1) {$\al_2$};
\node at (1.8,1) {$\al_1+\al_2$};
\node at (0.8,2) {$\al_1+2\al_2$};
\end{tikzpicture}
\]
In this case, the coefficient ring is $\OF=\zz$.

Second, we consider the ``normalized" root system of type $\tB_2$ given by the Cartan matrix
\[
C=\left[\begin{matrix} 2 & -\sqrt{2} \\ -\sqrt{2} & 2 \end{matrix}\right].
\]
In this case, the root system is no longer crystallographic:
\[
\begin{tikzpicture}
\node[inner sep=0pt] at (0,0)
    {\includegraphics[scale=0.75]{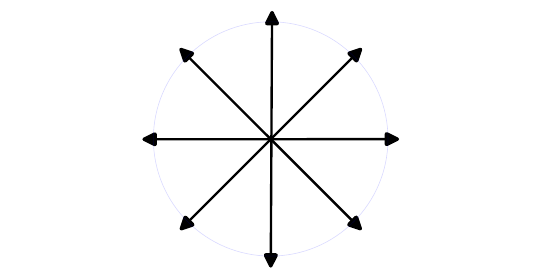}};

\node at (2,0) {$\al_1$};
\node at (-1.5,1) {$\al_2$};
\node at (2.1,1) {$\sqrt{2}\al_1+\al_2$};
\node at (1,1.8) {$\al_1+\sqrt{2}\al_2$};
\end{tikzpicture}
\]
The coefficient ring is $\OF=\zz(\sqrt{2})$ and the root lattice $\cl_r$ is the $\zz(\sqrt{2})$-span of $\al_1$ and $\al_2$.
Note that, while $\cl_r$ is a free $\zz$-module of rank 4, it is also dense (in the Euclidean sense) in $\rr^2$.
\end{ex}


\subsection{\it Coxeter diagrams and classification}
Any Coxeter group
\[
W=\langle s_1,\ldots,s_n \colon (s_is_j)^{m_{ij}}=1 \rangle
\]
can be geometrically realized as a root system as follows.
Let $\{\al_1,\ldots,\al_n\}$ and $\{\al_1^\vee,\ldots,\al_n^\vee\}$ be bases of Euclidean spaces $\rr^n$ and $(\rr^n)^*$ where
\[
(\al_i,\al_i)=1\quad\text{and}\quad \al_j^\vee(\al_i)=-2\cos (\pi/m_{ij}).
\]
Then the orthogonal subgroup of $\GL(\rr^n)$ generated by simple reflections
\[
s_{i}(x)=x-\al_i^\vee(x)\,\al_i,\quad x\in \rr^n,
\]
gives rise to a root system whose Weyl group is isomorphic to $W$.
We call such realization the {\em finite (normalized) real root system corresponding to} $W$.

Finite (normalized) real root systems are classified by {\em Coxeter diagrams} \cite[\S2]{Humphreys}.
Similar to Dynkin diagrams, its vertices are in 1-1 correspondence with the set of simple roots $\{\al_1,\ldots,\al_n\}$.
Different vertices $\al_i$ and $\al_j$ are connected by an edge only if $m_{ij}\ge 3$.
In addition, if $m_{ij}\ge 4$, then the respective edge is labelled by the Coxeter exponent $m_{ij}$.
The respective Cartan matrix $C=[\al_j^\vee(\al_i)]_{ij}$ is symmetric.
When $W$ is not the Weyl group of a finite crystallographic root system,
the finite real root systems are classified by the Coxeter diagrams of the following types
(here $m=5$ or $m\ge 7$):
\[
\begin{tabular}{ll}
$\xymatrix@C=2.5em{
\tI_2(m): & \node \lulab{1}  \ar@{-}[r]|-m  & \node \lulab{2}
}$
& \qquad
$\xymatrix@C=2.5em{
\tH_3:  & \node \lulab{1} \ar@{-}[r]|-5  & \node \lulab{2} \edge{r}
& \node \lulab{3}
}$
\\
$\xymatrix@C=2.5em{
\tH_4:& \node \lulab{1} \ar@{-}[r]|-5 & \node \lulab{2} \edge{r}
& \node \lulab{3} \edge{r} & \node \lulab{4}
}$
&
\end{tabular}
\]

Observe that $\tI_2(2)$ and $\tI_2(3)$ correspond to finite crystallographic root systems of type $\tA_1\times \tA_1$ and $\tA_2$, respectively.
Moreover, $\tI_2(4)$ and $\tI_2(6)$ can be viewed as the normalized crystallographic root systems of type $\tB_2$ ($=\tC_2$) and $\tG_2$, respectively.  Indeed, these are obtained by scaling one of the simple roots by $2\cos (\pi/m)=\sqrt{2}$ and $\sqrt{3}$, respectively (See Example \ref{Example:B2}).  In the following examples, we give an explicit construction of the normalized and polarized real root systems of the dihedral group $\tI_2(m)$.

\begin{ex}\label{Ex:I_5 real-root-system}
We realize $\tI_2(m)$ as the Weyl group of a real root system as follows.
Under the standard identification of $\cc\simeq \rr^2$, let $\zeta$ denote the $2m$-th primitive root of unity.
Consider the finite set of vectors $\RS=\langle \zeta\rangle\subset \rr^2$.
For any $\al\in\RS$, the orthogonal reflection (in complex coordinates) is given by
\[
s_\al(z)=(-\al \bar{\al}^{-1})\,\bar{z}.\]
In particular, if $\al=\zeta^j$, then
\[
s_{\zeta^j}(z)=-\zeta^{2j}\bar z=\zeta^{2j+m}\bar z.
\]
Hence $s_{\zeta^j}(\zeta^i)=\zeta^{2j+m-i}$ and the set $\RS$ forms a root system with Weyl group
$\tI_2(m)=\langle s_\al\colon \al\in\RS\rangle$.
Moreover, since
\[
s_{\zeta^j}(\zeta^i)=\zeta^{2j+m-i}=\zeta^i-(\zeta^{i-j}+\zeta^{j-i})\zeta^j,
\]
we conclude that the co-roots satisfy $(\zeta^j)^\vee (\zeta^i)=\zeta^{i-j}+\zeta^{j-i}$.  If we set $$\tau_m:=2\cos(\pi/m)=\zeta+\zeta^{-1}$$ and define simple roots

\[
\SR=\{\al_1=1,\al_2=\zeta^{m-1},\}
\]
then $\al_1^\vee(\al_2)=\al_2^\vee(\al_1)=-\tau_m$ and the simple reflections act by
\[
s_{i}(\al_j)=\al_j+\tau_m\,\al_i
\]
for $i\neq j$.  Thus, $\RS$ is a normalized root system.  It is easy to see that $s_1,s_2$ are Coxeter generators of the Weyl group of $\tI_2(m)$.
The coefficient ring
\[
\OF=\langle \zeta^{i-j}+\zeta^{j-i} \rangle_{i,j}=\zz(\tau_m)
\]
is the ring of integers in the real cyclotomic field $\qq(\tau_m)$.

If $m=5$, then $\tau_5=\tfrac{1+\sqrt{5}}{2}$ is a root of $x^2-x-1=0$.  If $m=7$, then $\tau_7$ is a root of $x^3-x^2-2x+1=0$.
The diagram below corresponds the normalized root system of type $\tI_2(5)$.
 \[
\begin{tikzpicture}
\node[inner sep=0pt] at (0,0)
    {\includegraphics[scale=0.75]{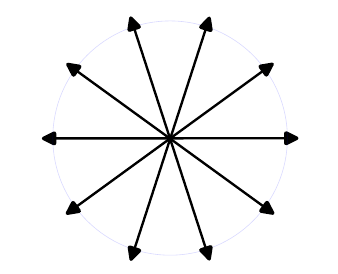}};

\node at (2,0) {$\al_1$};
\node at (-1.7,1) {$\al_2$};
\node at (2.3,1) {$\tau_5\al_1+\al_2$};
\node at (-1.1,1.8) {$\al_1+\tau_5\al_2$};
\node at (1.1,1.8) {$\tau_5\al_1+\tau_5\al_2$};
\end{tikzpicture}
\]

In contrast to normalized root system, we next consider polarized root systems of type $\tI_2(m)$ in the case that $m$ is even.  When $m$ is even, the action of $\tI_2(m)$ on the normalized root system $\RS$ has two orbits which are generated by the simple roots $1$ and $\zeta^{m-1}$ respectively.  We define a new root system of type $\tI_2(m)$ by scaling the vectors in the orbit of $\zeta^{m-1}$ by $\tau_m^{-1}$ and leaving the orbit of 1 unchanged.  Any root system isomorphic to this rank 2 root system is called the \emph{polarized root system} of type $\tI_2(m)$.  In this case, we have simple roots
\[
\SR=\{\al_1=1,\al_2=\tau_m^{-1}\cdot \zeta^{m-1}\}.
\]
with $\al_1^\vee(\al_2)=-1$, $\al_2^\vee(\al_1)=-(\tau_m^2)$ and the coefficient ring $\OF=\zz(\tau_m^2)$.  The simple reflections act on the simple roots by
\[s_1(\alpha_2)=\alpha_1+\alpha_2 \quad\text{and}\quad s_2(\alpha_1)=\alpha_1+\tau_m^2\alpha_2.\]
The $\tI_2(m)$-orbits of $\alpha_1$ and $\alpha_2$ are called the long and short roots respectively.  Note that the crystallographic root systems of type $\tB_2$ and $\tG_2$ are polarized root systems corresponding to $m=4$ and $m=6$.  Intuitively, polarized root systems can be thought of as root systems that are ``as close as possible" to crystallographic.

If $m=8$, then $\tau_8=\sqrt{2+\sqrt{2}}$ and $\OF=\zz(\tau_m^2)=\zz(\sqrt{2})$.
The diagram below corresponds the polarized root system of type $\tI_2(8)$.
 \[
\begin{tikzpicture}
\node[inner sep=0pt] at (0,0)
    {\includegraphics[scale=0.75]{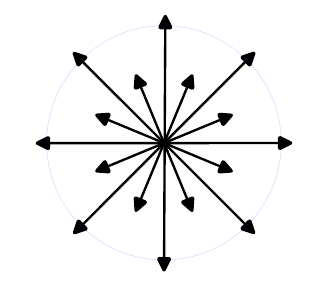}};

\node at (2,0) {$\al_1$};
\node at (-1.125,0.5) {$\al_2$};
\node at (1.625,0.5) {$\al_1+\al_2$};
\node at (-2,1.5) {$\al_1+(2+\sqrt{2})\al_2$};
\node at (3.5,1.5) {\hspace{1in}};
\end{tikzpicture}
\]
\end{ex}
We remark that when $m$ is odd, then any non-normalized root system of type $\tI_2(m)$ must be of at least rank 3.



\section{Nil-Hecke rings and twisted group algebras}\label{sec:nilHecke}

In this section we define the nil-Hecke ring associated to a root datum $\RS\hra \cl^\vee$ and give several examples.
We then realize the nil-Hecke ring as a sub-ring of a twisted group algebra and as a subring of certain operators on the polynomial ring of the root lattice.
These two presentations provide important computational tools for dealing with nil-Hecke rings.
Our exposition follows the original construction of the nil-Hecke ring given by Kostant and Kumar in \cite{KK86} and the generalized version of it in \cite{CZZ}.

\subsection{\it The nil-Coxeter and the nil-Hecke ring}
We start with the following definition

\begin{dfn}\label{Def:nil_coxeter_ring}
Let $W$ be a Coxeter group with generating set $\{s_1,\ldots,s_n\}$
subject to the Coxeter relations $(s_is_j)^{m_{ij}}=1$ for $1\leq i,j\leq n$.
The {\em nil-Coxeter ring} $\NN$ associated to $W$ is the unital ring with free generators $\{X_i\colon i=1,\ldots,n\}$
modulo the {\em nilpotent relations}
\[
X_i^2=0\quad \text{for all $1\leq i\leq n$},
\]
and {\em braid relations}
\[
\underbrace{X_iX_jX_i\ldots}_{m_{ij}}=\underbrace{X_jX_iX_j\ldots}_{m_{ij}}\quad \text{for all $i\neq j$}.
\]
\end{dfn}

In other words, $\NN$ is obtained from the integral group ring of the Coxeter group $W$ by substituting the relation $s_i^2=1$ with $X_i^2=0$.
The following lemma follows immediately from the definition of nil-Coxeter ring.

\begin{lem}\label{lem:Nil-Coxeter}
Let $\NN$ denote the nil-Coxeter ring associated to a Coxeter group $W$ and let $w\in W$.
If $w=s_{i_1}s_{i_2}\ldots s_{i_r}=s_{j_1}s_{j_2}\ldots s_{j_r}$ are two reduced expressions for $w$, then
\[
X_{i_1}X_{i_2}\ldots X_{i_r}=X_{j_1}X_{j_2}\ldots X_{j_r}.
\]
In particular, the product of generators $X_w:=X_{i_1}X_{i_2}\ldots X_{i_r}$
does not depend on the reduced expression of $w$ and is well-defined.

Furthermore, if $u,v\in W$, then
\[
X_u\cdot X_v=
\begin{cases}
X_{uv} & if\quad \ell(u)+\ell(v)=\ell(uv),\\
0 & if\quad \ell(u)+\ell(v)<\ell(uv).
\end{cases}
\]
\end{lem}

Observe that Lemma~\ref{lem:Nil-Coxeter} implies that the nil-Coxeter ring $\NN$ is a free $\zz$-module with basis $\{X_w\}_{w\in W}$.

Now suppose we are given a root datum $\RS \hra \cl^\vee$
in the sense of either~\ref{Def:root_datum_int} or~\ref{Def:root_datum_real}.
By definition the lattice $\cl$ is a free $\OF$-module and, hence, we may consider the symmetric algebra
\[
S(\cl):=\Sym_{\OF}(\cl).
\]
If we fix a $\OF$-basis $\{x_1,\ldots,x_n\}$ of $\cl$, then $S(\cl)$ is isomorphic to the polynomial ring ${\OF}[x_1,\ldots,x_n]$.
Note that the respective Coxeter group $W$ acts on the lattice $\cl$ and, therefore, it acts on the symmetric algebra $S(\cl)$ which we denote simply by $S$.

\begin{dfn}\label{Def:nil_Hecke_ring}
Given a root datum $\RS\hra \cl^\vee$ and a basis $\{x_1,\ldots,x_n\}$ of $\cl$,
the {\em nil-Hecke ring} associated to $\RS\hra \cl^\vee$ is defined to be the left $S$-module
\[
\NH:=S\otimes_{\OF} \NN
\]
equipped with multiplication induced by the relations called the {\em affine relations}:
\begin{equation}\tag{Aff}
X_i\, x_j = s_i(x_j)\, X_i + \al_i^\vee(x_j)\quad \text{for all }1\leq i,j\leq n.
\end{equation}
\end{dfn}
Note that the definition of the nil-Hecke ring does not depend on the choice of basis of $\cl$.
From the affine relations~(Aff) it follows that $\NH$ is a free left $S$-module with basis $\{X_w\}_{w\in W}$.
The relations~(Aff) also induce the following relations in $\NH$:

\begin{prop}\label{prop:general_aff}
For any $w\in W$ and $p\in \cl$ we have the following relation in $\NH$:
\[
X_w p = w(p) X_w + \sum_{w' \xra{\be} w} \be^{\vee}(p) X_{w'},
\]
where $u\xra{\be} w$ means $w=u s_\be$ for some root $\be\in \RS_+$ with $\ell(w)=\ell(u)+1$.
\end{prop}

\begin{proof}
We proceed by induction on $\ell(w)$.
If $\ell(w)=1$, then the proposition follows directly from~(Aff).
Now assume that the relation holds for all $u<w$ and
write $w=s_iu$ from some simple generator $s_i$ where $\ell(w)=\ell(u)+1$.
By induction, we have
\begin{align*}
X_w p = X_i \cdot (X_u p)&=X_i \cdot \Big(u(p) X_u + \sum_{u'\xra{\gamma} u} \gamma^{\vee}(p) X_{u'}\Big)\\
&=w(p) X_w + \al_i^{\vee}(u(p)) X_u + \sum_{u'\xra{\gamma} u} \gamma^{\vee}(p)\, X_i\cdot X_{u'}.
\end{align*}
First note that $u\xra{u^{-1}(\al_i)}w$ and, hence the second summand belongs to the desired sum in the relation.
Now suppose that $w'\xra{\be} w$ and $u\neq w'$ and hence $w=w's_{\be}=s_iu$.
This implies $u=(s_iw')s_{\be}$ with $\ell(u)=\ell(s_iw')+1$ and thus $s_iw'\xra{\be} u$.
Conversely, if $u'\xra{\gamma} u$, then $u'=s_iw'$ from some $(w')\xra{\gamma} w$
if and only if $\ell(s_iu')=\ell(u')+1$.
Hence, the summand in the expression above yields the desired result.
\end{proof}

\begin{ex}
Consider the Coxeter group of type $\tA_3$ and let $w=s_1s_2s_3$.
Then for any $p\in \cl$ we have
\[
X_w p= w(p) X_w+ \al_3^{\vee}(p)X_{s_1s_2}+ s_3(\al_2)^{\vee}(p) X_{s_1s_3}+ s_3s_2(\al_1)^{\vee}(p) X_{s_2s_3}.
\]
\end{ex}

\begin{rem}
Observe that an embedding of intermediate lattices $\cl \hra \cl'$ induces
the embedding of the respective polynomial rings $S(\cl)\to S(\cl')$.
This further induces an embedding of the respective nil-Hecke rings $\NH \hra \NH'$.
This embedding becomes an isomorphism after inverting the order$|\cl'/\cl|$.
In other words, we get an isomorphism between $\NH$ and $\NH'$ if we extend the coefficient ring $\OF$ to $\OF[\tfrac{1}{|\cl'/\cl_r|}]$.
However, the rings $\NH$ and $\NH'$ are not necessarily isomorphic over $\OF$.
\end{rem}

\begin{ex}
Let $\RS$ be the root system of type $\tA_n$.
The Weyl group $W$ is the permutation group generated by $\{s_1,\ldots,s_n\}$ subject to the Coxeter relations:
\[
s_i^2=1, \quad s_is_{i+1}s_i=s_{i+1}s_is_{i+1}\quad\text{and}\quad s_is_j=s_js_i\quad\text{if $|i-j|>1$.}
\]
The nil-Coxeter group $\NN$ is generated by $\langle 1,X_1,\ldots,X_n\rangle$ with the analogous relations:
\begin{equation}\label{Eqn:nil-coxeter-typeA}
X_i^2=0,\quad X_iX_{i+1}X_i=X_{i+1}X_iX_{i+1},\quad\text{and}\quad X_iX_j=X_jX_i\quad\text{if $|i-j|>1$.}
\end{equation}

Consider the root datum $\RS\hra \cl_w^\vee$ corresponding to the weight lattice (it corresponds to the simply connected simple group $\SL_{n+1}$).
The weight lattice has a basis of fundamental weights $\{\om_1,\ldots,\om_n\}$.
Using the notation from Section~\ref{Section_typeA_example}, we can write $\om_i=e_1+e_2+\cdots+e_i$,
where $\{e_1,\ldots,e_{n+1}\}$ denotes the standard coordinate basis of $\rr^{n+1}$.
The nil-Hecke ring is then given by
\[
\NH_w=\zz[\om_1,\ldots,\om_n]\otimes_\zz \NN,
\]
where $\NN$ is the respective nil-Coxeter ring, and
the affine relations (here we assume $\om_0=\om_{n+1}=0$) are
\begin{equation}\label{eq:affr}
\begin{cases}
X_i\om_j=\om_j X_i& \text{if }i\neq j,\\
X_i\om_i=(\om_{i-1}-\om_i+\om_{i+1}) X_i + 1 & \text{otherwise.}
\end{cases}
\end{equation}

Observe that $e_1=\om_1$, $e_i=\om_i-\om_{i-1}$ for $i=2,\ldots, n$, and $e_{n+1}=-\om_n$
(here we assume $e_1+\ldots+e_{n+1}=0$).
Substituting into~\eqref{eq:affr} we obtain
\[
X_ie_j=e_j X_i,\text{ if }j\neq i,i+1,\; \text{ and } X_ie_i=e_{i+1}X_i+1,\; X_ie_{i+1}=e_iX_i-1
\]
which can be rewritten as
\begin{equation}\label{eq:afftypeA}
X_ie_j-e_{s_i(j)}X_i=\de_{i,j}-\de_{i+1,j}.
\end{equation}

If we substitute the polynomial ring $\zz[\om_1,\ldots,\om_n]$ by $\zz[e_1,\ldots,e_{n+1}]$
and forget the relation $e_1+\ldots +e_{n+1}=0$,
we then obtain the classical definition of the {\em nil-Hecke~algebra} generated by $\{e_1,\ldots,e_{n+1},X_1,\ldots,X_n\}$
which satisfies the nil-Coxeter relations given in~\eqref{Eqn:nil-coxeter-typeA}
and the additional affine relations in~\eqref{eq:afftypeA}.

Consider now the root datum $\RS\hra \cl_r^\vee$ corresponding to the root lattice instead of the weight lattice
(this corresponds to the adjoint simple group $\PGL_{n+1}$).
The root lattice has a basis of simple roots $\SR=\{\al_1,\ldots,\al_n\}$, and in terms of the standard basis of $\rr^{n+1}$
we have $\al_i=e_i-e_{i+1}$.
The nil-Hecke ring is then
\[
{\NH}_r=\zz[\al_1,\ldots,\al_n]\otimes_{\zz}\NN
\]
with the additional affine relations:
\[
\begin{cases}
X_i\al_j=(\al_i+\al_j)X_i-1 & \text{if }|i-j|=1,\\
X_i\al_j=\al_jX_i &  \text{if }|i-j|>1,\\
X_i\al_i=-\al_i X_i + 2 & \text{otherwise}.
\end{cases}
\]
By definition ${\NH}_r$ is a proper subring of ${\NH}_w$.
For example, take $n=1$. Then
\[
{\NH}_r=\zz[\al]\Big\langle
X \colon
{\scriptsize\begin{matrix}
X^2=0,\\ X\al=-\al X+2
\end{matrix}}
\Big\rangle \stackrel{\al\mapsto 2\om}\longrightarrow
\zz[\om]\Big\langle X \colon
{\scriptsize\begin{matrix}
X^2=0,\\ X\om=-\om X+1
\end{matrix}}
\Big\rangle={\NH}_w.
\]
Observe that the ring ${\NH}_r$ is not isomorphic to ${\NH}_w$ (over $\zz$),
as they are not even isomorphic modulo $2$ (${\NH}_r$ is commutative, but ${\NH}_w$ is not).
However, they become isomorphic after inverting $2$.
\end{ex}

Next, we consider an example of non-crystallographic type.

\begin{ex}
Consider the normalized real root system of type $\tI_2(5)$ (see Example~\ref{Ex:I_5 real-root-system}).
In this case $\OF=\zz(\tau_5)$ where $\tau_5=2\cos(\tfrac{\pi}{5})=\tfrac{1+\sqrt{5}}{2}$.

The nil-Hecke ring ${\NH}_r$ for $\RS\hra \cl_r^\vee$ (the root lattice) is the quotient of
\[
{\OF}[\al_1,\al_2]\langle 1, X_1,X_2\colon X_1^2=X_2^2=0,\; X_1X_2X_1X_2X_1=X_2X_1X_2X_1X_2 \rangle
\]
by (Aff) relations: ($i\neq j$, $i,j=1,2$)
\[
\begin{cases}
X_i \al_{i} = -\al_i X_i + 2,\\
X_i\al_j =(\al_j+\tau_5\al_i)X_i -\tau_5.
\end{cases}
\]

The nil-Hecke ring ${\NH}_w$ for $\RS\hra \cl_w^\vee$ (the weight lattice with $\al_i=2\om_i-\tau_5\om_j$) is the quotient of
\[
{\OF}[\om_1,\om_2]\langle 1, X_1,X_2\colon X_1^2=X_2^2=0,\; X_1X_2X_1X_2X_1=X_2X_1X_2X_1X_2 \rangle
\]
by (Aff) relations: ($i\neq j$, $i,j=1,2$)
\[
\begin{cases}
X_i \om_{i} = (\tau_5 \om_j-\om_i) X_i + 1,\\
X_i\om_j =\om_jX_i.
\end{cases}
\]

Direct computations show that $|\cl_w/\cl_r|=5$. So the inclusion ${\NH}_r\hra {\NH}_w$ becomes a ring isomorphism only after inverting $5$.
\end{ex}


\subsection{\it The nil-Hecke subring of a localized twisted group algebra}
Given a finite real root system, its Coxeter group and the associated root/weight lattice,
we introduce the localized twisted group algebra and show that the nil-Hecke ring is natural subring of this algebra.
The action of the twisted group algebra on the symmetric algebra over a root datum gives
an important representation of the nil-Hecke ring.
In this subsection we provide an overview of these constructions and explain its basic properties.

Consider a  root datum $\RS\hra \cl^\vee$ with coefficient ring $\OF$.
Let $R$ be a commutative $\OF$-algebra. Let $S:=\Sym_R(\cl)$ be the associated graded symmetric $R$-algebra.
The Coxeter group $W$ acts on $\cl$ by $\OF$-linear automorphisms and, hence,
it acts on $S$ by $R$-linear ring automorphisms.

\begin{dfn}
We define the {\em twisted group algebra} $S_W$ to be the free left $S$-module
\[
S_W:=S\otimes_R R[W].
\]
\end{dfn}

Each element of $S_W$ can be written as a $S$-linear combination
\[
\sum_{w\in W} p_w\de_w,
\]
where $\{\de_w\}_{w\in W}$ is the standard basis of the group ring $R[W]$ and $p_w\in S$ are coefficients.
For the identity element $e\in W$ we set $\de_e=1$.
We define the ring multiplication on $S_W$ to be the one coming from $R[W]$, however, we require it to satisfy the twisted commuting relation with respect to coefficients:
\[
w(p)\,\de_w =\de_w\, p,\quad  p\in S,\; w\in W.
\]
For example, if $p,q\in S$ and $w,v\in W$, then
\[
p\,\de_w \cdot q\,\de_v=p\,w(q)\,\de_{wv}.
\]
Let $\SR=\{\al_1,\ldots,\al_n\}\subset\cl$ be the set of simple roots.
We further assume that each $\al_i\in \SR$ is regular in $S$ (in other words, $\al_i$ is a non-zero element which is not a zero-divisor).
Let $\al\in \RS$.  Since $\al=w(\al_i)$ for some $w\in W$ and $\al_i\in\SR$, we have $\al$ is also regular in $S$.
A symmetric algebra $S$ in which every root is regular is called $\RS$-regular.

\begin{rem} (cf. \cite[Lemma~2.2]{CZZ})
The regularity of $\al_i$ holds if either the root datum does not contain an irreducible component of simply-connected type $\mathrm{C}$
or if $2$ is regular in $R$.
Indeed, in the simply-connected type $\tC_n$ we have $\al_n=-2\om_{n-1}+2\om_n \equiv 0\mod 2$.
\end{rem}

\begin{dfn}
Let $Q:=S\left(\tfrac{1}{\al},\al\in \RS\right)$ denote the localization of the $\RS$-regular algebra $S$ at all $\al\in\RS$.  We define the \emph{ localized twisted group algebra} to be \[Q_W:=Q \otimes_R R[W].\]
\end{dfn}

Observe that the canonical ring inclusion $S\hra Q$ induces the ring inclusion $S_W \hra Q_W$.

We now realize the nil-Coxeter and nil-Hecke rings as sub-rings of $Q_W$.
For each simple root $\al_i$, define the element
\[
X_i:=\tfrac{1}{\al_i}(\de_e-\de_{s_i})\in Q_W
\]
and let $\NN$ denote the sub-ring of $Q_W$ generated by $\{X_i\}_{i=1,\ldots,n}$.
It is proved in \cite{KK86} that the generators
$\{X_i\}_{i=1,\ldots,n}$ satisfy the nilpotent and braid relations.
Hence, if $w=s_{i_1}\ldots s_{i_r}$ is a reduced expression, we can define
\[
X_w:=X_{i_1}\ldots X_{i_r}
\]
with $X_w$ being independent of choice of reduced expression of $w$.
\begin{prop}\label{prop:nil_coxeter_in_Q_w}
The sub-ring $\NN\subseteq Q_W$ is isomorphic to the nil-Coxeter ring of $W$ as defined in Definition~\ref{Def:nil_coxeter_ring}.
\end{prop}

\begin{proof}
It suffices to show that the set $\{X_w\}_{w\in W}$ is linearly independent over $S$
and is therefore a $\zz$-basis of $\NN$ (following Lemma~\ref{lem:Nil-Coxeter}).
Let $w\in W$ and fix a reduced word $w=s_{i_1}\cdots s_{i_r}$.
Then expanding $X_w$ in the $\{\de_u\}_{u\in W}$ basis of $Q_W$ yields:
\begin{align*}
X_w&=\tfrac{1}{\al_{i_1}}(\de_e-\de_{s_{i_1}})\cdots \tfrac{1}{\al_{i_r}}(\de_e-\de_{s_{i_r}})\\
&=\tfrac{(-1)^{\ell(w)}}{\al_{i_1}s_{i_1}(\al_{i_2})\cdots (s_{i_1}s_{i_2}\cdots s_{i_{r-1}}(\al_{i_r}))}\, \de_w+\sum_{u<w} c_{w,u}\, \de_u.
\end{align*}
In particular, the transition matrix from $\{X_u\}_{u\in W}$ to $\{\de_u\}_{u\in W}$ is upper triangular with respect to Bruhat order
and is invertible over $Q$.
Hence $\{X_u\}_{u\in W}$ is a $Q$-basis of $Q_W$ which implies that $\{X_u\}_{u\in W}$ is $S$-linearly independent.
\end{proof}
Note that the denominator of the coefficient $c_{w,w}$ in the proof above is equal to the product of roots in $\RS^+\cap w(\RS^-)$.  This set of roots is called the inversion set of $w\in W$.  We now define
\[
\NH:=S\otimes_{R} \NN\subset Q\otimes_{Q} Q_W=Q_W.
\]
\begin{prop}\label{prop:nil_Hecke_in_Q_w}
The sub-ring $\NH\subseteq Q_W$ is isomorphic to the nil-Hecke ring of $W$ as defined in Definition \ref{Def:nil_Hecke_ring}.
\end{prop}

\begin{proof}
To prove that $\NH$ is isomorphic to the nil-Hecke ring,
it suffices to show that the relations (Aff) hold.
Indeed, let $p\in\cl.$  Then for each $1\leq i\leq n$, we have
\begin{align*}
X_i p-s_i(p)X_i&=\tfrac{1}{\al_i}(\de_e-\de_{s_i})p-\tfrac{s_i(p)}{\al_i}(\de_e-\de_{s_i}) \\
&= \tfrac{p}{\al_i}\de_e-\tfrac{s_i(p)}{\al_i}\de_{s_i}- \tfrac{s_i(p)}{\al_i}\de_{e}+\tfrac{s_i(p)}{\al_i}\de_{s_i}=
\al_i^\vee(x)\de_e. \qedhere
\end{align*}
\end{proof}


\subsection{\it The polynomial representation}
From now on, we will view the nil-Hecke ring as an $S$-subalgebra of $Q_W$.
We define the $\circ$-action of the localized twisted group algebra $Q_W$ on $Q$ by $R$-linear endomorphisms via
\begin{equation}\label{Eqn:circle_action}
p\de_w \circ q:=pw(q).
\end{equation}
Under this action the element $X_i$ acts on $Q$ as the divided-difference operator
\[
X_i\circ q =\frac{q-s_i(q)}{\al_i}, \quad q\in Q.
\]
Divided-difference operators preserve the polynomial subalgebra $S$,
so the $\circ$-action gives a homomorphism of $R$-modules
\[
Q_W\to \Hom_R(S,Q)
\]
which restricts to a homomorphism of $R$-algebras which we denote by:
\[
\rho\colon \NH\to \End_R(S).
\]
Specifically, for any $y=\sum_{w\in W} a_w X_w\in\NH$ and $p\in S$, we have
\[
\rho(y)(p):=\sum_{w\in W} a_w (X_w\circ p)\in S.
\]

\begin{lem}\label{lem:bullet_properties}
Let $p,q\in S$ and $1\leq i\leq n$.  Then the following are true:

\begin{enumerate}
\item[(i)] If $p\in\cl$, then $X_i\circ p=\al_i^\vee(p)$
\item[(ii)] The operator $X_i$ satisfies the twisted Leibniz rule:
\[
X_i\circ(pq)=(X_i\circ p)\, q+s_i(p)\,X_i\circ q
\]
\end{enumerate}
\end{lem}

\begin{proof}
If $p\in\cl$, then
\[
X_i\circ p=\tfrac{1}{\al_i}(p-(p-\al_i^\vee(p)\al_i)) = \al_i^\vee(p).
\]
This proves part (1).  For part (2), we have
\begin{align*}
X_i\circ(pq)&= \tfrac{1}{\al_i}(pq-s_i(pq))= \tfrac{1}{\al_i}(pq-s_i(p) q + s_i(p) q - s_i(pq))\\
&=\tfrac{1}{\al_i}((p-s_i(p)) q + s_i(p) (q - s_i(q)))=(X_i\circ p)\, q+s_i(p)\,X_i\circ q. \qedhere
\end{align*}
\end{proof}

Note that Lemma~\ref{lem:bullet_properties}.(i) implies that the affine relation (Aff) can be rewritten as
\[
X_i\, p = s_i(p) X_i +X_i\circ p,\quad p\in \cl.
\]
It is easy to see that $X_i\circ p =0$ if and only if $s_i(p)=p$.
Otherwise, for any homogenous polynomial $p\in S$, we have
\[
\deg (X_i\circ p)=\deg(p)-1
\]
More generally, we have $\deg(X_w\circ p)=\deg(p)-\ell(w)$ and
hence if $\deg(p)=\ell(w)$, then $X_w\circ p\in R$.
This leads to the following lemma:

\begin{lem}\label{lem:faith}
Let $w_0$ be the element of maximal length in $W$.
Suppose there exists a homogeneous polynomial $u_0\in S$ of degree $\ell(w_0)$ such that
$X_{w_0}\circ u_0$ is a regular element in $R$.
Then the representation $\rho$ is faithful.
\end{lem}

\begin{proof}
Observe that $X_{w_0}\circ u_0=0$, if $u_0$ belongs to the ideal of $W$-invariants of $S$.
Indeed, as in the proof of \cite[Thm. 7.10]{CZZ1}, we may assume that $X_{w_0}\circ u_0$ is invertible in $R$.
Suppose $y=\sum_w a_w X_w$, $a_w\in S$ acts by $0$ (i.e, $\rho(y)=0$). Then for any $v\in W$,
\[
(y\cdot X_v)\circ u_0=\sum_w a_w (X_{wv}\circ u_0)=0.
\]
Taking $v=w^{-1}w_0$, we obtain that each $a_w=0$ and hence $y=0$.
\end{proof}

\begin{rem}
Suppose that $R=\OF$.
Consider the ideal $J$ in $\OF$ generated by all elements $X_{w_0}\circ p$ where $p\in S^{\ell(w_0)}:=\Sym_{\OF}^{\ell(w_0)}(\cl)$.
Then the assumption of the lemma translates into assuming that $J$ contains a regular element.  The ideal $J$ can be viewed as the generalization of the torsion index given in \cite{CZZ1}.  Indeed, in the crystallographic case $\OF=\zz$, the only generator $g_1$ of $J$ gives precisely the torsion index of the root datum.
In general, $\OF$ is not necessarily a PID, but a Dedekind domain.
In the latter case the ideal $J$ can be uniquely factorized into the product of prime ideals
\[
J=\mathfrak{p}_1^{r_1}\ldots \mathfrak{p}_s^{r_s}.
\]
Each of the prime ideals $\mathfrak{p}_i$ in this factorization can be thought of as a torsion prime (prime ideal) of the root datum.
\end{rem}


\section{The coproduct structure}\label{sec:coprod}

In this section we equip the localized twisted group algebra $Q_W$
with a co-product structure that plays an important role in the Schubert calculus of flag varieties.
This coproduct naturally restricts to nil-Hecke subrings as well as induces the respective augmented version. We discuss several recursive formulas for its structure coefficients and provide various examples of computations.

\subsection{\it The coproduct on twisted group algebras} We follow \cite[\S8]{CZZ} and \cite{KK86}.
Consider  $Q_W\otimes_{Q} Q_W$ over $Q$ (as left $Q$-modules).
Specifically, we have
\[
q\cdot (z\otimes z')= qz\otimes z'=z\otimes q z'.
\]
Define the coproduct
\[
\DI\colon Q_W \to Q_W \otimes_{Q} Q_W
\]
to be the morphism of left $Q$-modules given by
\[
\DI(q \de_w)=q \de_w \otimes \de_w= \de_w \otimes q\de_w.
\]
By definition the co-product $\DI$ is a morphism of $R$-algebras,
is $Q$-linear, coassociative, and cocommutative with counit $\eps\colon Q_W\to Q$ given by $q\de_w\mapsto q$.
Therefore, $Q_W$ is a cocommutative coalgebra in the category of left $Q$-modules.

We equip the tensor algebra $Q_W \otimes_{Q} Q_W$ with the following product.
\[
(q \de_w\otimes \de_v)\odot (q' \de_{w'}\otimes \de_{v'}):=q w(q')\, \de_{ww'}\otimes \de_{wv'w^{-1}v}.
\]
It can be shown that this product gives rise to a ring structure on $Q_W \otimes_{Q} Q_W$
and for any $q\de_w, z,z'\in Q_W$, we have
\[
(q\de_w\otimes \de_w)\odot (z\otimes z')=q\de_w z\otimes \de_w z'.
\]
In the next lemma we show that the ring structure is compatible with the twisted product on $Q_W$.
\begin{lem}
The coproduct $\DI\colon Q_W \to Q_W \otimes_{Q} Q_W$ is a ring homomorphism.
\end{lem}

\begin{proof}
Clearly the coproduct is linear, so it suffices to check the compatibility of multiplication.
Let $q_w\de_w, q_v\de_v\in Q_W$.  Then
\[
\DI(q_w\de_w\cdot q_v\de_v)=\DI(q_w w(q_v)\de_{wv})=q_w w(q_v) \de_{wv}\otimes \de_{wv}.
\]
On the other hand,
\begin{align*}
\DI (q_w\de_w)\odot \DI(q_v\de_v) &=(q_w\de_w\otimes\de_w)\odot (q_v\de_v\otimes\de_v)\\
&=q_ww(q_v)\de_{wv}\otimes \de_{www^{-1}v}=q_ww(q_v)\de_{wv}\otimes \de_{wv}. \qedhere
\end{align*}
\end{proof}

Next, we study how the co-product $\DI$ restricts to the nil-Hecke ring $\NH\subset Q_W$, but before we do this,
we consider a different $S$-basis of $\NH$.
For any $1\leq i\leq n$, define the {\em push-pull elements}
\[
Y_i:=\tfrac{1}{\al_i}(\de_{s_i}-\de_e).
\]
For purposes of notation, we set $Y_e:=\de_e=1$.
It is clear that $Y_i=-X_i$ and hence the set $\{Y_1,\dots, Y_n\}$ satisfy the nilpotent and braid relations.
For any reduced expression $w=s_{i_1}\ldots s_{i_r}$, we define
\[
Y_w:=Y_{i_1}\ldots Y_{i_r}.
\]
The push-pull elements $Y_w$ are independent of choice of reduced expression and $Y_w=(-1)^{\ell(w)} X_w$.
Propositions~\ref{prop:nil_coxeter_in_Q_w} and~\ref{prop:nil_Hecke_in_Q_w} imply that
the set $\{Y_w\}_{w\in W}$ is an $S$-basis of the nil-Hecke ring $\NH$.
The analogous affine relations to those given in Proposition~\ref{prop:general_aff} are
\begin{equation}\label{eqn:general_aff2}
Y_w\, p = w(p)\, Y_w - \sum_{w'\xra{\be} w} \be^{\vee}(p)\, Y_{w'}, \quad p\in\cl.
\end{equation}

\subsection{\it The coproduct on nil-Hecke rings}
The theorem below describes the restriction of the co-product $\DI$ to $\NH$
and the recursive structure of the co-product with respect to the push-pull basis.  This formula first appears in \cite{KK86} and has been subsequently been used in \cite{BR15} and \cite{GK21}.
For $w\in W$, we denote its left and right descent sets as:
\begin{align*}
D_L(w)&:=\{ s_i \colon \ell(s_iw)=\ell(w)-1\}, \text{ and}\\
D_R(w)&:=\{ s_i \colon \ell(ws_i)=\ell(w)-1\}.
\end{align*}

\begin{thm}\label{thm:coproduct_recursion}
The co-product map $\DI:Q_W\to Q_W\otimes_{Q} Q_W$ restricts to a ring homomorphism of $S$-modules
\[
\DI\colon \NH \to \NH\otimes_{S} \NH.
\]
Moreover, for any $i\leq n$, we have
\[
\DI(Y_i)=Y_i\otimes 1+1\otimes Y_i +\al_i Y_i\otimes Y_i,
\]
and for any $w\in W$ and $s_i\in D_{L}(w)$,  if we write
\[
\DI(Y_w)=\sum_{v,u\in W} p_{u,v}^w \ Y_u\otimes Y_v,
\]
then the coefficients satisfy the recursion
\[
p_{u,v}^w=\begin{cases}
Y_i\circ(p_{u,v}^{s_iw})&\text{if $s_i\notin D_L(u)\cup D_L(v)$}\\
Y_i\circ(p_{u,v}^{s_iw})+s_i(p_{s_iu,v}^{s_iw})&\text{if $s_i\in D_L(u)\setminus D_L(v)$}\\
Y_i\circ(p_{u,v}^{s_iw})+s_i(p_{u,s_iv}^{s_iw})&\text{if $s_i\in D_L(v)\setminus D_L(u)$}\\
Y_i\circ(p_{u,v}^{s_iw})+s_i(p_{s_iu,v}^{s_iw}+p_{u,s_iv}^{s_iw})+\al_is_i(p_{s_iu,s_iv}^{s_iw})&\text{if $s_i\in D_L(u)\cap D_L(v)$}
\end{cases}
\]
with $p_{e,e}^e=1$.
\end{thm}

\begin{proof}
Since $\{Y_w\}_{w\in W}$ forms a $Q$-basis of $Q_W$,
it suffices to show that the coefficients $p_{u,v}^w$ are all elements of $S$ and thus
$\DI(Y_w)\in \NH\otimes_{S} \NH$.
We first consider the case that $w=s_i$.  Indeed, we have
\begin{align*}
\DI(Y_i)&=\DI\big(\tfrac{1}{\al_i}(\de_{s_i}-\de_e)\big)\\
&=\tfrac{1}{\al_i}(\de_{s_i}\otimes\de_{s_i}-\de_e\otimes\de_e)\\
&=\tfrac{1}{\al_i}\big(\de_{s_i}\otimes(\de_{s_i}-\de_e)+(\de_{s_i}-\de_{e})\otimes\de_e\big)\\
&=\tfrac{1}{\al_i}\big((\de_{s_i}-\de_e)\otimes(\de_{s_i}-\de_e)+(\de_{s_i}-\de_{e})\otimes\de_e+\de_e\otimes(\de_{s_i}-\de_{e})\big)\\
&=\al_i(Y_i\otimes Y_i)+ Y_i\otimes 1 + 1\otimes Y_i.
\end{align*}
Hence, $\DI(Y_i)\in\NH\otimes_S \NH$.

Now consider $w\in W$ with $\ell(w)\geq 2$, and let $w=s_iw'$ from some $s_i\in D_L(w)$.
Write
\[
\DI(Y_{w'})=\sum_{u,v} p_{u,v}^{w'}\ Y_u\otimes Y_v
\]
and assume for the sake of induction that $\DI(Y_{w'})\in \NH\otimes_{S} \NH$, or equivalently, each $p_{u,v}^{w'}\in S$.
Then
\begin{align*}
\DI(Y_w)&=\DI(Y_i)\odot\DI(Y_{w'})\\
&=\tfrac{1}{\al_i}(\de_{s_i}\otimes\de_{s_i}-\de_e\otimes\de_e)\odot\big(\sum_{u,v} p_{u,v}^{w'}\ Y_u\otimes Y_v\big)\\
&= \sum_{u,v}\ \tfrac{1}{\al_i}\big(s_i(p_{u,v}^{w'})\ (\de_{s_i} Y_u\otimes\de_{s_i} Y_v) - p_{u,v}^{w'}\ (Y_u\otimes Y_v) \big)\\
&= \sum_{u,v}\ \tfrac{s_i(p_{u,v}^{w'})}{\al_i}\Big( (\de_{s_i}-\de_e)Y_u\otimes(\de_{s_i}-\de_e)Y_v +(\de_{s_i}-\de_e)Y_u\otimes Y_v\\
&\qquad\qquad + Y_u\otimes (\de_{s_i}-\de_e) Y_v  + Y_u\otimes Y_u\Big)-\tfrac{p_{u,v}^{w'}}{\al_i}(Y_u\otimes Y_u)\\
&= \sum_{u,v}\ \al_i\cdot s_i(p_{u,v}^{w'})\ (Y_iY_u\otimes Y_iY_v) + s_i(p_{u,v}^{w'})\ (Y_iY_u\otimes Y_v + Y_u\otimes Y_iY_v)\\
&\qquad\qquad + Y_i\circ p_{u,v}^{w'}\ (Y_u\otimes Y_v).
\end{align*}
Since the action of $\al_i s_i$, $s_i$ and $Y_i$ each preserve $S$,
we have $\DI(Y_w)\in \NH\otimes_{S} \NH$.
By induction, $\DI(\NH)\subseteq \NH\otimes_{S} \NH$.

Suppose that $Y_u=Y_iY_z$ in $\NH$. Since $Y_i^2=0$ and $Y_u=Y_{s_iz}$ in $\NH$, $\ell(z)<\ell(u)$ and $u=s_iz$ so $s_i\in D_L(u)$.
The recursive formula for the coefficients $p_{u,v}^w$ then follows directly from the calculation above.
\end{proof}

\begin{rem}\label{rem:X_w_coproduct}
Observe that $\DI(X_w)$ in terms of the basis $\{X_u\otimes X_v\}_{u,v\in W}$ can be readily calculated from $\DI(Y_w)$.
Indeed, if
\[
\DI(Y_w)=\sum_{v,u\in W} p_{u,v}^w \ Y_u\otimes Y_v,
\]
then
\begin{align*}
\DI(X_w)&=\sum_{v,u\in W} (-1)^{\ell(w)}p_{u,v}^w \ Y_u\otimes Y_v\\
&=\sum_{v,u\in W} (-1)^{\ell(w)+\ell(u)+\ell(v)}p_{u,v}^w \ X_u\otimes X_v.
\end{align*}
\end{rem}

\begin{ex}\label{ex:rk2co}
We give a calculation of $\DI(Y_w)$ using Theorem~\ref{thm:coproduct_recursion}.
Let $\RS$ denote a root system of rank at least 2, and consider the root datum $\RS\hra \cl^{\vee}_r$
corresponding to the root lattice.
Since the co-product is co-commutative, we have that $p_{u,v}^w=p_{v,u}^w$ for all $w,u,v\in W$.
With this symmetry in mind, define for $u\neq v$
\[
Y(u,v):=Y_u\otimes Y_v + Y_v\otimes Y_u, \; \text{ and }Y(u):=Y_u\otimes Y_u.
\]
Let $w=s_2s_1$.  Then
\begin{align*}
\DI(Y_w)=\al_2s_2(\al_1)\, Y(w)&+ s_2(\al_1)\, Y(w, s_1) + \al_2\, Y(w,s_2)\\
 &+ Y(w,e) + \al_2^{\vee}(\al_1)\, Y(s_1)+ Y(s_1,s_2).
\end{align*}
Note that $p^w_{s_2,s_2}=0$ and, hence, $Y(s_2)$ does not appear in the summand above.
The following table gives values of $s_2(\al_1)$ and $\al_2^{\vee}(\al_1)$ for specific root systems:
\medskip
\[
\begin{tabular}{ccc}
$\RS$ & $s_2(\al_1)$ & $\al_2^{\vee}(\al_1)$\\ \hline
$\tA_2$ & $\al_1+\al_2$ & 1\\
$\tB_2$ & $\al_1+2\al_2$ & 2\\
$\tB_2$ (normalized) & $\al_1+\sqrt{2}\al_2$ & $\sqrt{2}$\\
$\tI_2(5)$ (normalized) & $\al_1+\tau_5 \al_2$ & $\tau_5$
\end{tabular}
\]
\smallskip
\end{ex}

\begin{ex}\label{ex:coproduct_recursion}
We can use the recursion formula in Theorem~\ref{thm:coproduct_recursion} to compute individual coefficients $p_{u,v}^w$.
Consider the root system of $\tA_3$.  Here we will use $``i"$ to denote $s_i$ (i.e. $1231=s_1s_2s_3s_1$).
\begin{align*}
p_{13,13}^{1231}&=Y_1\circ p_{13,13}^{231}+ 2s_1(p_{3,13}^{231})+ \al_1s_1(p_{3,3}^{231})\\
&= Y_1\circ Y_2(p_{13,13}^{13}) + 2s_1(Y_2\circ (p_{3,13}^{31}))+0\\
&= Y_1\circ Y_2\circ(\al_1\al_3) + 2s_1(Y_2\circ (\al_3)).
\end{align*}
Applying Lemma~\ref{lem:bullet_properties} gives:
\begin{align*}
p_{13,13}^{1231}&= Y_1\circ((Y_2\circ\al_1)\al_3+s_2(\al_1) Y_2\circ \al_3) + 2s_1(Y_2\circ (\al_3))\\
&=\al_1^{\vee}(\al_3+\al_1+\al_2) + 2 \al_2^{\vee}(\al_3)\\
&= -1+2=1.
\end{align*}
\end{ex}

The next corollary records some immediate consequences of Theorem~\ref{thm:coproduct_recursion}
on the co-product structure constants $p_{u,v}^w$.

\begin{cor}
Let $w\in W$ and let
\[
\DI(Y_w)=\sum_{v,u\in W} p_{u,v}^w \ Y_u\otimes Y_v.
\]
Then the following are true:

\begin{enumerate}
\item[(i)]
If $p_{u,v}^w\neq 0$, then $u,v\leq w$ and
$p_{u,v}^w$ is a homogenous polynomial in the simple roots $\Delta=\{\al_1,\ldots,\al_n\}$ of degree $\ell(u)+\ell(v)-\ell(w)$.
\item[(ii)]
$\displaystyle p_{w,w}^w=\prod_{\al\in \RS^+\cap w\RS^-} \al$ (i.e. the product of inversions of $w$).
\item[(iii)]
$p_{w,e}^w=p_{e,w}^w=1$.
\end{enumerate}
\end{cor}


\subsection{\it The augmented coproduct}\label{S:augmented_coproduct} We now follow \cite[\S10]{CZZ}.
Recall that $\circ$-action of $\NH$ on $S$ defined in~\eqref{Eqn:circle_action}
gives rise to a representation $\rho\colon \NH\to \End_R(S)$.
Suppose that conditions of Lemma~\ref{lem:faith} hold,
so that the representation $\rho$ is faithful and
hence we can identify $\NH$ with its image $\rho(\NH)$.
Since $S$ is a graded algebra (by degree), there is an augmentation map
$\veps\colon  S\to S^{0}_{\cl}\simeq R$ given by $\veps(p):=p(0)$.
The augmentation map induces a map from the nil-Hecke ring
\[
\veps_*\colon \NH\to \Hom_R(S,R),\quad y\mapsto \veps \circ y,
\]
where $y\in \NH=\rho(\NH)$ is viewed as an endomorphism in $\End_R(S)$.
Let $\II$ denote the kernel of $\veps$.

\begin{lem}
The kernel of $\veps_*$ coincides with the right ideal $\II\NH$, where $\NH$ is viewed as a left $S$-module.
\end{lem}

\begin{proof}
Recall that $\{X_w\}_{w\in W}$ forms a basis of $\NH$ as a left $S$-module and let $y=\sum_w a_w\cdot  X_w\in\NH$.
Then $y\in \ker\veps_*$  if and only if
\[
\veps_*(y)(p)=\sum_w a_w(0)\cdot (X_w\circ p)(0)=0
\]
for each $p\in S$.
Let $u_0\in S$ denote the polynomial given in Lemma~\ref{lem:faith} and $w_0$ denote the maximal element in $W$.
For any $v\in W$, let $g_v:=X_{v^{-1}w_0}\circ u_0$.
Since $r:=X_{w_0}\circ u_0$ is regular in $R$ and $(X_w\circ u_0)(0)=0$ unless $w\neq w_0$,
we have $\veps_*(y)(g_v)=a_v(0)r$.
Therefore, if $y\in \ker\veps_*$, then $a_w(0)=0$ (so $a_w\in\II$) for each $w$.
\end{proof}

Let $\veps\NH$ denotes the image of $\veps_*$.
Then $\veps \NH \simeq \NH/\II\NH\simeq R\otimes_{S} \NH$ where the action on $S$ on $R$
is given by multiplication by $\veps(p)$ for any $p\in S$.
Observe that $\veps \NH$ is not an $R$-algebra but it is isomorphic to the nil-Coxeter ring $\NN$ as an $R$-module.
There is an $R$-linear coproduct $\DI^\veps$ on $\veps \NH$ defined by
\[
\DI^\veps(g)(p\otimes q)=g(pq)
\]
where $g\in \veps \NH$ and $p,q\in S$.
Here we view
\[
\DI^\veps(g)\in \veps\NH\otimes_R \veps\NH\simeq \rho(\NH)\otimes_R \rho(\NH)\subseteq \End_R(S\otimes_R S).
\]
The augmented co-product map fits into the commutative diagram
\[
\xymatrix{
\NH \ar[r]^{\veps_*} \ar[d]_{\DI}& \veps \NH \ar[d]^{{\DI}^\veps} \\
\NH\otimes_{S} \NH \ar[r]^{\veps_*^{\otimes 2}}& \veps\NH\otimes_R \veps\NH
}
\]
For any $w\in W$, let $\veps Y_w:=\veps_*(Y_w)$ (resp. $\veps X_w:=\veps_*(X_w))$.
Then the augmented coproduct satisfies
\[
{\DI}^\veps(\veps Y_w)=\sum_{\substack{u,v\in W\\ \ell(w)=\ell(u)+\ell(v)}} p_{u,v}^w\  \veps Y_u\otimes \veps Y_v
\]
where the $p_{u,v}^w$ are degree 0 coefficients found in Theorem~\ref{thm:coproduct_recursion}.
By Remark~\ref{rem:X_w_coproduct}, we also have  that
\[
{\DI}^\veps(\veps X_w)=\sum_{\substack{u,v\in W\\ \ell(w)=\ell(u)+\ell(v)}} p_{u,v}^w\  \veps X_u\otimes \veps X_v
\]
and hence the map $\veps\NH\to \veps\NH$ given by $\veps Y_w\mapsto \veps X_w$ is a co-algebra isomorphism over $R$.
Observe that the analogous bijection sending $Y_w\mapsto X_w$ on $\NH$ is not a co-algebra homomorphism, but is $S$-linear.
For convenience when working with augmented structure coefficients, we define
\[
c_{u,v}^w:=\begin{cases} p_{u,v}^w & \text{if $\ell(w)=\ell(v)+\ell(u)$} \\ 0 & \text{otherwise.} \end{cases}
\]
Hence, we have
\[
{\DI}^\veps(\veps Y_w)=\sum_{u,v\in W} c_{u,v}^w\  \veps Y_u\otimes \veps Y_v.
\]


\subsection{\it The expanded coproduct formula}
In Example~\ref{ex:coproduct_recursion}, we see that the twisted Leibniz formula
is used in the application of Theorem~\ref{thm:coproduct_recursion} for computing the structure coefficients $p_{u,v}^w$.
In this section, we give a non-recursive formula for the coefficients $p_{u,v}^w$ in terms of bounded bijections studied in \cite{BR15}.
To begin, let $[n]=\{1,2,\ldots,n\}$ and, following~\cite[\S4]{CZZ},
we first give a generalized Leibniz formula a sequence $\ii:=(i_1,\ldots,i_r)\in [n]^r$.
For any subset $E\subseteq [r]$, let $\ii_E$ denote the corresponding subsequence of $\ii$.
In other words, if $E=\{j_1<\cdots<j_k\}$, then
\[
\ii_E:=(i_{j_1},\ldots,i_{j_k}).
\]
Define the operator
\[
Y_{\ii}:=Y_{i_1}\circ \cdots \circ Y_{i_r}.
\]
We also, for any pair of subsets $E_1,E_2\subseteq [r]$, we define $p_{E_1,E_2}^\ii\in S$ as
\begin{equation}\label{Eqn:BS_coeff_operator}
p_{E_1,E_2}^\ii:=B_1 \cdots B_r(1)
\end{equation}
with the operator $B_j\colon S\to S$ defined by
\[
B_j(p):=\begin{cases}
\al_{i_j}s_{i_j}(p) & \text{if } j\in E_1\cap E_2,\\
Y_{i_j}\circ p & \text{if } j\notin E_1\cup E_2,\\
s_{i_j}(p) & \text{otherwise.}
\end{cases}\]
Observe that if $E_1\cap E_2=\emptyset$ and $E_1\cup E_2\neq [m]$, then $p_{E_1,E_2}^{\ii}=0$ since $Y_i\circ(1)=0$.  The coefficients $p_{E_1,E_2}^{\ii}$ are sometimes called Bott-Samelson coefficients since they can be identified with Schubert structure constants of a Bott-Samelson variety and have studied in this context in \cite{Wi04}, \cite{BR15} and \cite{GK21}.

The next lemma follows from repeated applications of Lemma \ref{lem:bullet_properties}.(ii) (see \cite[Lemma~4.8]{CZZ}):
\begin{lem}
Let $\ii:=(i_1,\ldots,i_r)\in [n]^r$.
For any $p,q\in S$ we have
\[
Y_{\ii}\circ(pq)=\sum_{E_1,E_2\subseteq [m]}p_{E_1,E_2}^{\ii} (Y_{\ii_{E_1}}\circ p)\cdot (Y_{\ii_{E_2}}\circ q).
\]
\end{lem}
Dualizing we obtain the corresponding formula for the coproduct on the nil-Hecke ring $\NH$ (cf. \cite[Prop.~9.5]{CZZ}):
\[
\DI(Y_{\ii})=\sum_{E_1,E_2\subseteq [r]} p_{E_1,E_2}^{\ii}\, Y_{\ii_{E_1}}\otimes Y_{\ii_{E_2}}.
\]

If $\ii$ corresponds to a reduced expression of $w\in W$ where $r=\ell(w)$,
we obtain the following formula for the coproduct in the nil-Hecke ring $\NH$:
\[
\DI(Y_w)=\sum_{u,v\in W} p_{u,v}^w\, Y_u\otimes Y_v,
\]
where
\begin{equation}\label{eqn:BS_coeff}
p_{u,v}^w=\sum\, p_{E_1,E_2}^{\ii},
\end{equation}
where the sum is over all pairs $(E_1,E_2)\subseteq[r]^2$ such that $\ii_{E_1}, \ii_{E_2}$ are reduced expressions for $u,v$ respectively.
The expression~\eqref{eqn:BS_coeff} is closely related to the recursive formula for $p_{u,v}^w$ given in Theorem \ref{thm:coproduct_recursion}.
In particular, each summand $p_{E_1,E_2}^{\ii}$ corresponds to a sequence of operators $B_j$ that is consistent with a path in the recursion.


\subsection{\it The formula in terms of bounded bijections}
Note that the coefficients $p_{E_1,E_2}^{\ii}$ can be further expanded
by applying Lemma \ref{lem:bullet_properties} to the operators $B_j$ when $j\in E_1\cap E_2$.
We state this expansion in terms of bounded bijections which are defined as follows

\begin{dfn}
Let $K,L$ be subsets of $[r]$ such that $|K|+|L|=r$.
We say a bijection $$\phi\colon K\to [r]\setminus L$$ is \emph{bounded} if $\phi(k)<k$ for all $k\in K$.
If $K=\emptyset$ and $L=[r]$, then we consider $\phi\colon K\to K$ a \emph{vacuous} bounded bijection.
\end{dfn}

For each bounded bijection $\phi\colon K\to [r]\setminus L$ and $k \in [r]$, define the set
\[
L(k):=L\cup \phi(K\cap [k]).
\]
Now given a sequence $\ii=(i_1,\ldots,i_r)\in [n]^r$ and a bounded bijection $\phi\colon K\to [r]\setminus L$,  we define a root $\al^{\ii}_k$ for each $k\in [r]$ as follows.  Let
\[
\al^{\ii,\phi}_k:=w_k(\al_{i_k})\quad \text{where}\quad w_k:=\displaystyle \buildrel\longrightarrow \over{\prod_{\substack{j\in L(k)\\ \phi(k)<j<k}}} s_{i_j}
\]
where we set $\phi(k):=0$ if $k\notin K$ and the product $\displaystyle\buildrel\longrightarrow \over{\prod}$ is taken in the natural order induced by the sequence $[r]$.
If $K=\emptyset$, then we set $\al^{\ii,\phi}_k:=1$ for all $k\in [r]$.

The following theorem is proved in \cite{BR15} using a careful application of Lemma \ref{lem:bullet_properties} on the definition of $p_{E_1,E_2}^\ii$ given in~\eqref{Eqn:BS_coeff_operator}.

\begin{thm}\label{thm:bbijections1}
Given a sequence $\ii=(i_1,\ldots,i_r)$ and subsets $E_1,E_2\subseteq [m]$, the coefficient
\[
p_{E_1,E_2}^\ii=\sum_{(K,\phi)} \Big(\prod_{k\in K} \al^\vee_{i_{\phi(k)}} (\al_{\ell}^{\ii, \phi})\Big)\Big(\prod_{\ell\in(E_1\cap E_2)\setminus K}\al_{\ell}^{\ii, \phi}\Big)
\]
where the summation is over all pairs $(K,\phi)$ such that
\begin{enumerate}
\item $K$ is a subset of $E_1\cap E_2$ such that $|K|+|E_1\cup E_2|=r$; and
\item $\phi\colon K\to [r]\setminus (E_1\cup E_2)$ is a bounded bijection.
\end{enumerate}
\end{thm}

Combining Theorem~\ref{thm:bbijections1} with the formula~\eqref{eqn:BS_coeff} gives a non-recursive formula for the structure coefficients $p_{u,v}^w$.  In the next theorem, we observe that several terms cancel in this combined expression.  To do that we need the following

\begin{dfn}
Given a sequence $\ii=(i_1,\ldots,i_r)$ and a subset $J\subseteq [r]$, we say that the subsequence $\ii_{J}$ is \emph{admissible} if $i_j\neq i_{j+1}$ for all $j\in J$.
We say a bounded bijection $\phi$ is \emph{$\ii$-admissible} if the subsequence $\ii_{L(k)}$ is admissible for all $k\in K$.
\end{dfn}

Note that if a sequence $\ii$ corresponds to a reduced word of some $w\in W$, then $\ii$ is admissible.

\begin{thm}\label{thm:bbijections2}
Let $w,u,v\in W$ and let $\ii=(i_1,\ldots,i_r)$ correspond to a reduced expression of $w$.  Then the coefficient $p_{u,v}^w$ can be computed as
\[
p_{u,v}^w=\sum_{(E_1,E_2)}\sum_{(K,\phi)} \Big(\prod_{k\in K} \al^\vee_{i_{\phi(k)}} (\al_{k}^{\ii, \phi})\Big)\Big(\prod_{\ell\in(E_1\cap E_2)\setminus K}\al_{\ell}^{\ii, \phi}\Big),
\]
where the first sum is taken over all pairs $(E_1,E_2)\subseteq[r]^2$ such that $\ii_{E_1}, \ii_{E_2}$
are reduced expressions for $u,v$ respectively and the second sum is over all pairs $(K,\phi)$ such that
\begin{enumerate}
\item $K$ is a subset of $E_1\cap E_2$ such that $|K|+|E_1\cup E_2|=r$; and
\item $\phi\colon K\to [m]\setminus (E_1\cup E_2)$ is an $\ii$-admissible bounded bijection.
\end{enumerate}
\end{thm}

A proof of Theorem~\ref{thm:bbijections2} can be found in \cite{BR15}.
The general idea is that terms corresponding to bounded bijections that are not $\ii$-admissible
can be ``canceled" when grouping terms across various pairs $(E_1,E_2)$ corresponding to reduced words of $u$ and $v$.

\begin{rem}
We remark that, even with the $\ii$-admissible reduction the sum given in Theorem~\ref{thm:bbijections2} is not manifestly positive.
In other words, it is possible that summands have negative coefficients.
In the case of crystallographic root datum, it is known for geometric reasons that
the coefficients $p_{u,v}^w$ are polynomials in the simple roots with non-negative coefficients.
\end{rem}

Observe that if $\ell(u)+\ell(v)=\ell(w)=r$, then condition~(1) in Theorem~\ref{thm:bbijections2} implies that
$K=E_1\cap E_2$ is the subset $K$ satisfying condition~(1).
Hence, for augmented coefficients $c_{u,v}^w$, we have the following simplification of Theorem \ref{thm:bbijections2}.

\begin{cor}\label{cor:bbijections3}
Let $w,u,v\in W$ and let $\ii=(i_1,\ldots,i_r)$ correspond to a reduced expression of $w$.
If $\ell(w)=\ell(u)+\ell(v)$, then the augmented coefficient
\[
c_{u,v}^w=\sum_{(E_1,E_2,\phi)}\Big(\prod_{k\in E_1\cap E_2} \al^\vee_{i_{\phi(k)}} (\al_{k}^{\ii, \phi})\Big)
\]
where the sum is over all triples $(E_1,E_2,\phi)$ such that
\begin{enumerate}
\item
$(E_1,E_2)\subseteq[r]^2$ such that $\ii_{E_1}, \ii_{E_2}$ are reduced expressions for $u,v$ respectively; and
\item
$\phi\colon E_1\cap E_2\to [r]\setminus (E_1\cup E_2)$ is an $\ii$-admissible bounded bijection.
\end{enumerate}
\end{cor}

\begin{ex}
Consider the root system of type $\tG_2$  which has Cartan matrix
\[
C=\left[\begin{matrix} 2 & -3 \\ -1 & 2 \end{matrix}\right].
\]
Let $w=s_1s_2s_1s_2$ and $u=v=s_1s_2$.
In this case, $\ii=(1,2,1,2)$ and $m=4$.
There are nine pairs $(E_1,E_2)\subseteq [4]^2$ that correspond to reduced words of $u,v$
which we highlight with over/underlines in the following diagram.
The boxed entries are the pairs $(E_1,E_2)$ that have bounded bijections.
\[
\framebox{$1\ 2\ \underline{\overline{3}}\ \underline{\overline{4}}$}
\quad\framebox{$\underline{1}\ 2\ \overline{3}\ \underline{\overline{4}}$} \quad \framebox{$\underline{1}\ \underline{2}\
\overline{3}\ \overline{4}$} \quad \framebox{$\overline{1}\ 2\ \underline{3}\ \underline{\overline{4}}$}\quad
\overline{\underline{1}}\ 2\ 3\ \underline{\overline{4}}
\]
\[
\overline{\underline{1}}\ \underline{2}\ 3\ \overline{4} \quad \framebox{$\overline{1}\ \overline{2}\ \underline{3}\ \underline{4}$}\quad \overline{\underline{1}}\
\overline{2}\ 3\ \underline{4} \quad \overline{\underline{1}}\ \overline{\underline{2}}\ 3\ 4.
\]

For $E_1=E_2=\{3,4\}$, (i.e. $1\ 2\ \underline{\overline{3}}\ \underline{\overline{4}})$, we have
\[
E_1\cap E_2=\{3,4\}\quad \text{and}\quad [4]\setminus(E_1\cup E_2)=\{1,2\}.
\]
There are two bounded bijections given by
\[
\phi_1\colon (3,4)\to(1,2)\qquad\phi_2:(3,4)\to(2,1).
\]
Theorem \ref{thm:bbijections1} implies that
\begin{align*}
p^{\ii}_{E_1,E_2}&=\al_{i_1}^{\vee}(\al_{i_3})\cdot \al_{i_2}^\vee(s_{i_3}(\al_{i_4}))+\al_{i_2}^\vee(\al_{i_3})\cdot\al_{i_1}^\vee(s_{i_2}s_{i_3}(\al_{i_4}))\\
&=\al_{1}^{\vee}(\al_{1})\cdot \al_{2}^\vee(s_{1}(\al_{2}))+\al_{2}^\vee(\al_{1})\cdot\al_{1}^\vee(s_{2}s_{1}(\al_{2}))\\ &=(-2)(1)+(3)(0)=-2.
\end{align*}
For $(E_1,E_2)=(\{1,4\},\{3,4\})$ or $(\{3,4\},\{1,4\})$, we have the single bounded bijection
\[
\phi_3\colon \{4\}\to \{2\}
\]
and
\[
p^{\ii}_{E_1,E_2}=\al_{i_2}^\vee(s_{i_3}(\al_{i_4}))=\al_{2}^\vee(s_{1}(\al_{2}))=1.
\]
For $(E_1,E_2)=(\{1,2\},\{3,4\})$ or $(\{3,4\},\{1,2\})$, the sets
\[
E_1\cap E_2=\emptyset\quad\text{and}\quad E_1\cup E_2=[4]
\]
and hence $p^{\ii}_{E_1,E_2}=1$.
Summing these expressions as in the formual~\eqref{eqn:BS_coeff} gives
\[
c^w_{u,v}=p^w_{u,v}=-2+0+(1+1)+(1+1)=2.
\]
Note that $\phi_1$ and the two copies of $\phi_3$ are not $\ii$-admissible.
Theorem~\ref{thm:bbijections2} implies that the corresponding terms cancel ($-2+(1+1)=0$) leaving the final sum for $p^w_{u,v}$ unchanged.

\smallskip

If we repeat this calculation for the normalized root system of type $\tI_2(5)$ with Cartan matrix
\[
C=\left[\begin{matrix} 2 & -\tau_5 \\ -\tau_5 & 2 \end{matrix}\right],
\]
then
\[
c^w_{u,v}=p^w_{u,v}=(-2)(\tau_5-1)+(-1)+(\tau_5-1)+(\tau_5-1)+1+1=1.
\]
Note that this calculation uses the relation $\tau_5^2-\tau_5-1=0$.
In this case, the cancelation of terms that are not $\ii$-admissible is
\[
(-2)(\tau_5-1)+(\tau_5-1)+(\tau_5-1)=0.
\]
Also note that the remaining $\ii$-admissible summands are not all positive.
\end{ex}


\subsection{\it Augmented coproduct formula in rank two}
Now we provide an application of Corollary~\ref{cor:bbijections3} to finite root systems of rank 2.
In this case, the Weyl group $W$ is of type $\tI_2(m)$ for some positive integer $m\geq 2$
and the respective Cartan matrix is given by
\begin{equation}\label{eqn:rank2_cartan_matrix}
C=\left[\begin{array}{cc}2& -a\\-b&2 \end{array}\right]
\end{equation}
where $a,b$ are any positive real numbers such that $\sqrt{ab}=2\cos(\pi/m)$.
If $m=2$, we set both $a=b=0$.
For any $r\leq m$, define
\[
u_r:=\underbrace{\cdots s_1s_2s_1}_{r}\quad \text{and}\quad v_r:=\underbrace{\cdots s_2s_1s_2}_{r}.
\]
If $1\leq r<m$, then $u_r,v_r$ are the two unique elements in $W$ of length $r$.
Otherwise $u_0=v_0=1$ and $u_m=v_m$ is the longest element in $W$.
From Corollary~\ref{cor:bbijections3} (also Theorem~\ref{thm:coproduct_recursion}),
it is not difficult to conclude that the augmented coefficients satisfy
\[
c^{u_{r+t}}_{u_r,u_t}=c^{v_{r+t+1}}_{v_{r+1},u_{t}}=c^{v_{r+t+1}}_{u_{r},v_{t+1}},\qquad
c^{v_{r+t}}_{v_r,v_{t}}=c^{u_{r+t+1}}_{u_{r+1},v_{t}}=c^{u_{r+t+1}}_{v_{r},u_{t+1}},
\]
and
\[
c^{v_{r+t}}_{u_r,u_{t}}=c^{u_{r+t}}_{v_{r},v_{t}}=0.
\]
Hence it suffices to analyze the coefficients $c^{u_{r+t}}_{u_r,u_{t}}$ and $c^{v_{r+t}}_{v_r,v_{t}}$.
Define the sequences $A_k$ and $B_k$ by
\[
A_k:=a B_{k-1}-A_{k-2}\quad \text{and}\quad  B_k:=b A_{k-1}-B_{k-2}
\]
where $A_0=B_0=0$ and $A_1=B_1=1$.
For $k\leq m$, define the binomial coefficients
\[
C(r,t):=\frac{A_{r+t}A_{r+t-1}\cdots A_1}{(A_rA_{r-1}\cdots A_1)(A_{t}A_{t-1}\cdots A_1)},
\]
\[
D(r,t):=\frac{B_{r+t}B_{r+t-1}\cdots B_1}{(B_rB_{r-1}\cdots B_1)(B_{t}B_{t-1}\cdots B_1)}.
\]
We remark that the recursive structure of these sequences is related of Chebyshev polynomials of second kind.
The following theorem is proved by Kitchloo in \cite[\S10]{Ki14} for crystallographic root systems,
by Berenstein and Kapovich in \cite{BK11} for normalized root systems and
the first author and Berenstein for real root systems in \cite{BR15}.
\begin{thm}\label{thm:rank2_binom}
Let $\cl$ be a root system of rank 2 with the Weyl group $W$ of type $\tI_2(m)$ and with the Cartan matrix~\eqref{eqn:rank2_cartan_matrix}.
For any $1\leq r,t\leq m$, we have
\[
c^{u_{r+t}}_{u_r,u_t}=C(r,t)\quad \text{and}\quad c^{v_{r+t}}_{v_r,v_t}=D(r,t).
\]
\end{thm}

\begin{proof}
We give an outline of the proof given in \cite{BR15} for the coefficients $c^{u_{r+t}}_{u_r,u_t}$
(the proof for $c^{v_{r+t}}_{v_r,v_t}$ is similar).
Let $\ii=(\ldots,2,1,2,1)$ denote the repeating sequence of length $m':=r+t$.
Let $(E_1,E_1)\subseteq [m']^2$ be a pair of subsets such that $\ii_{E_1},\ii_{E_2}$ are reduced expressions of $u_r, u_t$ respectively.
It can be shown that each such pair such admits at most one $\ii$-admissible bounded bijection
$\phi\colon E_1\cap E_2\to [m']\setminus(E_1\cup E_2)$.
Applying this fact to the formula in Corollary \ref{cor:bbijections3}, we get that
\[
c^{u_{m'}}_{u_r,u_t}=\al_{i_1}^\vee(v_{m'-2}(\al_1)) \cdot c_{v_{r-1},v_{t-1}}^{v_{m'-2}}+c_{u_{r-2},u_{t}}^{u_{m'-2}}+c_{u_{r},u_{t-2}}^{u_{m'-2}}.
\]
Inducting on $m'$ gives the formula
\[
c^{u_{m'}}_{u_r,u_t}=\al_{i_1}^\vee(v_{m'-2}(\al_1)) \cdot D(r-1,t-1)+C(r-2,t)+C(t,r-2).
\]
It can also be shown that
\[
\al_{i_1}^\vee(v_{m'-2}(\al_1))=\begin{cases} A_{m'}-A_{m'-2} &\text{if $m'$ is odd}\\ B_{m'}-B_{m'-2} &\text{if $m'$ is even.}\end{cases}
\]
A technical calculation (details can be found in \cite{BR15}) shows that the binomial $C(r,t)$ satisfies the same recursion and
hence  $c^{u_{m'}}_{u_r,u_t}=C(r,t)$.
\end{proof}

\begin{ex}
We consider two examples of root systems with with the Weyl group of type $\tI_2(m)$.
The first is the normalized root system  with
\[
a=b=\tau_m=2\cos(\pi/m)=\zeta+\zeta^{-1},
\]
where $\zeta$ denotes a $m$-th primitive root of unity.
A simple calculation shows that
\[
A_k=B_k=[k]_\zeta:=\zeta^{k-1}+\zeta^{k-3}+\cdots + \zeta^{3-k}+\zeta^{1-k}=\tfrac{\zeta^k-\zeta^{-k}}{\zeta-\zeta^{-1}},
\]
and
\[
c^{u_{r+t}}_{u_r,u_t}=c^{v_{r+t}}_{v_r,v_t}={\tbinom{r+t}{r}}_{\zeta}:=\frac{[r+t]_{\zeta}!}{[r]_{\zeta}!\cdot [t]_{\zeta}!},
\]
where $[r]_{\zeta}!:=[r]_{\zeta}[r-1]_{\zeta}\cdots [1]_{\zeta}$.
Note that $[m]_{\zeta}=0$ and hence $c^{u_{r+t}}_{u_r,u_t}=0$ if $r+t\geq m$.
The minimal polynomial of $\tau_m$ can be used to given an alternate expression of the sequence $\{A_k\}_{k\geq 0}$.
For example, if $m=7$, then $x^3-x^2-2x+1$ is the minimal polynomial of $\tau_7$ and the sequence $\{A_k\}_{k\geq 0}$ is
\[
(0,1,\tau_7, \tau_7^2-1, \tau_7^2-1, \tau_7, 1,0).
\]
If we treat $\zeta$ as an indeterminant, then $\displaystyle{\tbinom{r+t}{r}}_{\zeta}$ is the Gaussian binomial
used in the study of quantum groups and
the Gaussian integers $[k]_\zeta$ can be viewed as the characters of irreducible $\SL_2(\cc)$-representations.
Also note that if we set $\zeta=1$, then we recover classical binomial coefficients.
When $\zeta=1$, the corresponding Cartan matrix is for the infinite root system of the affine Lie group $\widehat{\SL}_2(\cc)$.

\smallskip

Next, we consider the polarized root system with Weyl group $\tI_2(m)$ where $m$ is even or $m=3$.
For the polarized root system, we set
\[
a=\tau_m^2=(\zeta+\zeta^{-1})^2\quad\text{and}\quad b=1.
\]
Here we have
\[
A_k=\begin{cases} [k]_{\zeta} &\text{if $k$ is odd} \\ [2]_\zeta\cdot [k]_{\zeta} & \text{if $k$ is even}\end{cases}\qquad\text{and}\qquad
B_k=\begin{cases} [k]_{\zeta} &\text{if $k$ is odd} \\ [2]_{\zeta}^{-1}\cdot [k]_{\zeta} & \text{if $k$ is even.}\end{cases}
\]
Theorem \ref{thm:rank2_binom} implies
\[
\displaystyle c^{u_{r+t}}_{u_r,u_t}=\begin{cases}\displaystyle [2]_\zeta\cdot{\tbinom{r+t}{r}}_{\zeta} &\text{if $r,t$ are both odd,} \\ \displaystyle {\tbinom{r+t}{r}}_{\zeta} & \text{otherwise.}\end{cases}
\]
and
\[
c^{v_{r+t}}_{v_r,v_t}=\begin{cases}\displaystyle [2]_{\zeta}^{-1}\cdot{\tbinom{r+t}{r}}_{\zeta} &\text{if $r,t$ are both odd,} \\
\displaystyle {\tbinom{r+t}{r}}_{\zeta} & \text{otherwise.}\end{cases}
\]
Note that polarized root systems include crystallogrpahic root systems of types $\tA_2$, $\tB_2$ and $\tG_2$.
These cases correspond to taking $m=3,4,$ and $6$ respectively
(we can also recover $\tA_1\times \tA_1$ if we allow $a=b=0$ with $m=2$).
For example, the crystallographic root system of type $\tG_2$ ($m=6$) has $a=[2]_{\zeta}^2=3$ and $b=1$.
In this case, the sequences $\{A_k\}_{k\geq 0}$ and $\{B_k\}_{k\geq 0}$ are respectively given by
\[
(0,1,3,2,3,1,0)\quad \text{and}\quad (0,1,1,2,1,1,0).
\]
Using these values, we compute augmented coefficients using Theorem \ref{thm:rank2_binom}.
For example:
\[
c^{u_4}_{u_1,u_3}=\tfrac{3\cdot 2\cdot 3\cdot 1}{(1)(2\cdot 3\cdot 1)}=3\quad \text{and}\quad c^{v_5}_{v_3,v_2}=\tfrac{1\cdot 1\cdot 2\cdot 1\cdot 1}{(2\cdot 1\cdot 1)(1\cdot 1)}=1.
\]
If we consider the normalized root system of type $\tG_2$, then $a=b=\sqrt{3}$ and the sequence $A_k=B_k$ is given by
\[
(0,1,\sqrt{3},2,\sqrt{3},1,0)
\]
and the coefficient
\[
c^{u_4}_{u_1,u_3}=\tfrac{\sqrt{3}\cdot 2\cdot \sqrt{3}\cdot 1}{(1)(2\cdot \sqrt{3}\cdot 1)}=\sqrt{3}.
\]
Note that the polarized (crystallographic) and normalized coefficients of $c^{u_4}_{u_1,u_3}$ are off by a scalar of $[2]_{\zeta}=\sqrt{3}$.
\end{ex}


\section{Dual twisted group algebras and nil-Hecke rings}\label{sec:twistedgr}
In this section, we study duals of the twisted group algebra, nil-Hecke ring and its connections to Schubert calculus.

\subsection{\it The group basis and the Schubert basis}
The localized ring $Q$ is a commutative unital ring and $Q_W$ is a free left $Q$-module of finite rank.
Furthermore, $Q_W$ is a $Q$-co-algebra,
with linear co-product ${\DI}\colon Q_W \to Q_W\otimes_Q Q_W$ and counit $\eps\colon Q_W\to Q$.
Then the $Q$-linear dual
\[
Q^*_W:=\Hom_Q(Q_W,Q)
\]
has a natural structure of $Q$-algebra with unit, given by dualizing the co-product and using the natural isomorphism
$Q_W^*\otimes_Q Q_W^*\simeq (Q_W\otimes_Q Q_W)^*$, and dualizing the co-unit,
(i.e, $\eps\colon Q_W\to Q$ seen as an element of $Q_W^*$ is the unit).
The algebra $Q_W^*$ is commutative if and only if $Q_W$ is co-commutative.
As a $Q$-vector space, addition and scaling on $Q_W^*$ is the usual point-wise operations
\[
(f+g)(\de_w)=f(\de_w)+g(\de_w),\quad\text{and}\quad (qf)(\de_w)=q(f(\de_w))
\]
for $f,g\in Q_W^*$, $w\in W$ and $q\in Q$.  The multiplication induced by $\DI$ is given by
\[
(f\cdot g)(\de_w)=f(\de_w)g(\de_w).
\]
Let $\{\psi_w\}_{w\in W}$ and  $\{\xi_w\}_{w\in W}$ denote the Kronecker dual bases to
$\{\de_w\}_{w\in W}$ and $\{Y_w\}_{w\in W}$ respectively.  In other words:
\[
\psi_w(\de_v)=\xi_w(Y_v):=\de_{w,v}.
\]
We call  $\{\psi_w\}_{w\in W}$ the \emph{group basis} and $\{\xi_w\}_{w\in W}$ the \emph{Schubert basis} of $Q_W^*$.
Using the group basis, the product structure can be interpreted as the component-wise multiplication
\begin{equation}\label{eqn:group_basis_product}
\Big(\sum_{w\in W} q_w\psi_w\Big)\cdot \Big(\sum_{w\in W} q_w' \psi_w\Big)=\sum_{w\in W}  q_wq_{w}'\psi_w.
\end{equation}
In terms of these bases, the unit in $Q_W^*$ is given by
\[
\xi_e=\sum_{w\in W}\psi_w.
\]


\subsection{\it Transformation matrices and coefficients}
We define the coefficients $d_{w,v}$ by writing $\xi_w$ in terms of the group basis
\[
\xi_w=\sum_{v\in W}d_{w,v} \psi_v.
\]
By definition we have
\[
\xi_w(\de_v)=d_{w,v}.
\]
We also define the analogous coefficients $c_{w,v}$ by expanding
\begin{equation}\label{eqn:c_wv_coeficients}
Y_w=\sum_{v\in W} c_{w,v}\de_v
\end{equation}
in the twisted group algebra $Q_W$.
By definition we have that
\[
\de_{w,v}=\xi_w(Y_v)=\sum_{u\in W}c_{v,u}\xi_w(\de_u)=\sum_{u,u'\in W}c_{v,u}d_{w,u'}.
\]
If we define matrices $C:=[c_{w,v}]_{w,v\in W}$ and $D:=[d_{w,v}]_{w,v\in W}$, then the equation above translates to
\[
D^\mathrm{t}=C^{-1}.
\]
The following proposition appears in \cite{KK86} and \cite{Ku96}.

\begin{prop}\label{prop:transistion_coefficients}
Let $w$, $v \in W$, $w=s_{i_1}\ldots s_{i_r}$ be a reduced word, and let $\ii:=(i_1,\ldots,i_r)$.
For any $k\leq r$ and subsequence $\jj=(j_1,\ldots, j_s)\subseteq \ii$, let $(j_1,\ldots, j_t)$ denote the prefix of $\jj$ that appears strictly before the entry $i_k$ in $\ii$ and define
\[
\be_{k,\jj}:=s_{j_1}\cdots s_{j_{t}}(\al_{i_k}).
\]
Then
\begin{equation}\label{eqn:cuv_formula}
c_{w,v}=(-1)^{\ell(w)-\ell(v)}\sum_{\jj\subseteq \ii }\Big( \prod_{k\in \ii} \be_{k,\jj}^{-1}\Big),
\end{equation}
where the sum is taken over all subsequences $\jj=(j_1,\ldots, j_s) \subseteq \ii$ for which
$v=s_{j_1}\ldots s_{j_s}$ (not necessarily reduced).

In particular, the following properties hold:
\begin{enumerate}
\item If $v\nleq w$, then $c_{w,v}=d_{v,w}=0$,
\item $\displaystyle c_{w,w}^{-1}=d_{w,w}=\be_{1,\ii}\cdots \be_{r,\ii}=\prod_{\al\in \RS^+\cap w\RS^-} \al$, and
\item $d_{e,v}=1$.
\end{enumerate}
\end{prop}

\begin{proof}
Let $w=s_{i_1}\ldots s_{i_r}$ be a reduced expression.
Then expanding $Y_w$ gives
\begin{align*}
Y_w=Y_{i_1}\ldots Y_{i_r}=\tfrac{1}{\al_{i_1}}(\de_{s_{i_1}}-\de_e)\ldots \tfrac{1}{\al_{i_r}}(\de_{s_{i_r}}-\de_e).
\end{align*}
For any subset $\jj=(j_1,\ldots, j_s)\subseteq \ii$ the coefficient of $\de_{s_{j_1}}\ldots \de_{s_{j_s}}$ in the expansion of $Y_w$ is
\[
(-1)^{\ell(w)-s}\prod_{k\in \ii} \be_{k,\jj}^{-1}.
\]
If $v=s_{j_1}\ldots s_{j_s}$, then $\ell(v)=s\hspace{-5pt}\mod 2$ which proves the first part of the proposition.
For the additional properties, note that the matrix $C$ is lower triangular with respect to the Bruhat order.
Parts~(1) and (2) now follow from the fact that $D^\mathrm{t}=C^{-1}$ and
applying the formula given in Equation~\eqref{eqn:cuv_formula} to $c_{w,w}$.
To prove part~(3), note that
\[
\xi_e=\sum_{v\in W}\psi_v
\]
is the identity element in $Q_W^*$ and, therefore, $d_{e,v}=1$ for all $v\in W$.
\end{proof}

For an example of Equation \eqref{eqn:cuv_formula}, see Example \ref{ex:smooth}.  We will later prove an analogous formula for the coefficient $d_{w,v}$ similar to Equation \eqref{eqn:cuv_formula}.
The coefficients above can be used to express the products of group basis elements with Schubert basis elements.

\begin{lem}\label{lem:group-schubert_product}
Let $u,v\in W$.  Then
\[\psi_u\cdot \xi_v=\sum_{w\in W}(c_{w,u}\cdot d_{v,u})\, \xi_w.\]
\end{lem}

\begin{proof}
Let
$\psi_u\cdot \xi_v=\sum_{w\in W}b_{u,v}^w\, \xi_w$
for some coefficients $b_{u,v}^w\in Q$ and note that $b_{u,v}^w=\psi_u\cdot \xi_v(Y_w)$.
If we write $\displaystyle Y_w=\sum_{w'\in W}c_{w,w'} \de_{w'}$, then we get
\[
b_{u,v}^w=\psi_u\cdot \xi_v\Big(\sum_{w'\in W}c_{w,w'} \de_{w'}\Big)=\sum_{w'\in W}c_{w,w'} \psi_u(\de_{w'})\cdot \xi_v(\de_{w'})=c_{w,u}\cdot d_{v,u}.\qedhere
\]
\end{proof}


\subsection{\it The Hecke action}
Next, we define an action of $Q_W$ on $Q_W^*$ called the Hecke action.
 This action plays an important role in understanding the structure of the dual $Q_W^*$.  Many results from this subsection also appear in \cite[Chapter 11]{Ku02}.

\begin{dfn}\label{Def:Hecke_action}
We define the  \emph{Hecke action} of $Q_W$ on $Q_W^*$ by
\[
q\de_v\bullet f(\de_w):=f(\de_w(q\de_v))=f(w(q)\de_{wv})=w(q)f(\de_{wv}).
\]
\end{dfn}

For the group basis $\{\psi_w\}_{w\in W}$ it can be shown that
\begin{equation}\label{eqn:Hecke_action_group_basis}
q\de_v\bullet \psi_w=wv^{-1}(q) \psi_{wv^{-1}}
\end{equation}
and
\begin{equation}\label{eqn:Hecke_action_group_basis2}
Y_i\bullet \psi_w=\frac{1}{w(\al_i)} (\psi_{ws_i}-\psi_w).
\end{equation}
In the next two lemmas we show how $Q_W$ acts on the Schubert basis of $Q_W^*$.
\begin{lem}\label{lem:bullet_Schuberts}
For any $v,w\in W$, we have
\[
Y_v\bullet \xi_w=\begin{cases} \xi_{wv^{-1}} & \text{if $\ell(w)=\ell(v)+\ell(wv^{-1})$}\\
0 & \text{otherwise}.
\end{cases}
\]
\end{lem}

\begin{proof}
Observe that the Hecke action of push-pull elements on the Schubert basis gives
\begin{equation*}
(Y_i\bullet \xi_w)(Y_u)=\xi_w(Y_uY_i)=\begin{cases} 1 & \text{if $w=us_i$ with $\ell(w)=\ell(u)+1$}\\
0 & \text{otherwise}
\end{cases}
\end{equation*}
and hence
\begin{equation}\label{eqn:bullet_dual_action_Y}
Y_i\bullet \xi_w=\begin{cases} \xi_{ws_i} & \text{if $s_i\in D_R(w)$}\\
0 & \text{otherwise}.
\end{cases}
\end{equation}
If $v=s_{i_1}\cdots s_{i_r}$ is a reduced word,
then applying formula~\eqref{eqn:bullet_dual_action_Y} to $Y_v:=Y_{i_1}\cdots Y_{i_r}$ proves the lemma.
\end{proof}

\begin{lem}\label{lem:bullet_group_basis}
Let $p\in\cl$ and $w\in W$.  Then
\begin{equation}\label{eqn:bullet_group_action1}
(p\de_e)\bullet\xi_w=w(p)\xi_w-\sum_{w\xra{\be}v} \be^{\vee}(p)\, \xi_v.
\end{equation}
Furthermore, for any $1\leq i\leq n$, we have
\[
(\de_{s_i})\bullet\xi_w=
\begin{cases}
 \xi_w- w(\al_i)\xi_{ws_i}-\sum_{ws_i\xra{\be}v} \be^{\vee}(\al_i)\, \xi_v &\text{if $s_i\in D_R(w)$},\\
\xi_w &\text{otherwise}.
\end{cases}
\]
\end{lem}

\begin{proof}
For the first part, write $$(p\de_e)\bullet\xi_w=\sum_{v\in W} c_v \xi_v$$
for some coefficients $c_v\in Q$.  By duality of bases and relation~\eqref{eqn:general_aff2}, we have
\begin{align*}
c_v=(p\de_e\bullet\xi_w)(Y_v)&=\xi_w(Y_vp)\\
&=\xi_w(v(p)\, Y_v - \sum_{v'\xra{\be} v} \be^{\vee}(p)\, Y_{v'}).
\end{align*}
This implies
\[
c_v=
\begin{cases} w(p) & \text{if $v=w$},\\
\be^\vee(p) & \text{if $w\xra{\be} v$},\\
0 & \text{otherwise.}
\end{cases}
\]
proving formula~\eqref{eqn:bullet_group_action1}.
As for the second formula note that $\de_{s_i}=\al_i Y_i+\de_e$ and, hence,
\[
\de_{s_i}\bullet \xi_w=\xi_w+ (\al_iY_i)\bullet \xi_w=\xi_w+ \al_i\de_e\bullet (Y_i\bullet \xi_w).
\]
The second part of the lemma follows
from formulas~\eqref{eqn:bullet_dual_action_Y} and~\eqref{eqn:bullet_group_action1}.
\end{proof}

\begin{prop}\label{prop:local_coefficients_prop2}
The local coefficients $d_{w,v}$ satisfy the recursive formula
\[
d_{w,vs_i}-d_{w,v}=
\begin{cases} v(\al_i)d_{ws_i,v} & \text{if $s_i\in D_R(w)$},\\
0 & \text{otherwise.}\end{cases}
\]
Furthermore, we have
\[
d_{s_i,v}=\om_i-v(\om_i)
\] where $\om_i$ denotes the $i$-th fundamental weight.
In particular, if $w\leq v$, then the local coefficient $d_{w,v}$ is a homogeneous polynomial in the simple roots
$\{\al_1,\ldots,\al_n\}$ of degree $\ell(w)$.
\end{prop}

\begin{proof}
First note that $d_{e,v}=1$ for all $v\in W$ and thus the proposition is true for $w=e$.
Suppose that $\ell(w)\geq 1$.
Formula~\eqref{eqn:Hecke_action_group_basis2} implies that the Hecke action
\[
Y_i\bullet\xi_w=\sum_{v\in W} d_{w,v}\, (Y_i\bullet\psi_v)=\sum_{v\in W} \frac{d_{w,v}}{v(\al_i)}\, (\psi_{vs_i}-\psi_v).
\]
The recursive formula for $d_{w,v}$ now follows from Lemma \ref{lem:bullet_Schuberts}.

We prove the second part of the proposition by induction.
First note that $d_{s_i,e}=0$ by Proposition \ref{prop:transistion_coefficients} part (1).
Also note that simple reflections act on their corresponding fundamental weights by  $s_i(\om_i)=\om_i-\al_i$.
Now assume $\ell(v)\geq 1$ and suppose that $s_i\in D_R(v)$.
Then the recursive formula gives
\begin{align*}
d_{s_i,v}=d_{s_i, vs_i}-v(\al_i)d_{e,v} &= (\om_i-vs_i(\om_i))-v(\al_i)\\ &= \om_i-v(s_i(\om_i)+\al_i)=\om_i-v(\om_i).
\end{align*}
Otherwise, if $s_i\notin D_R(v)$, then again by Proposition \ref{prop:transistion_coefficients}~part (4), we have
\begin{align*}
d_{s_i, vs_i}=d_{s_i,v}+v(\al_i)d_{e,v} &= (\om_i-v(\om_i))+v(\al_i) \\ &=\om_i-v(\om_i-\al_i)=\om_i-vs_i(\om_i).
\end{align*}
Finally, the fact that $d_{w,v}$ are homogeneous polynomials in the simple roots of degree $\ell(w)$ follows inductively
from the recursive formula and that $d_{e,v}=1$ for all $v\in W$.
\end{proof}


\subsection{\it Dual nil-Hecke ring}
In this section we focus on the dual to the nil-Hecke ring and
give several formulas for the structure constants with respect to the Schubert basis.
Define $\NH^*$ the $S$-submodule of $Q_W^*$ generated by the Schubert basis
$\{\xi_w\}_{w\in W}$.
It is easy to see that
\[
\NH^*=\Hom_{S}(\NH,S)
\]
is the $S$-linear dual to the nil-Hecke ring $\NH$.
Theorem \ref{thm:coproduct_recursion} implies that
\begin{equation}\label{eqn:Schubert_structure_constants}
\xi_u\cdot \xi_v=\sum_{w\in W}p_{u,v}^w\, \xi_w
\end{equation}
with all coefficients $p_{u,v}^w\in S$.
Hence $\NH^*$ is an $S$-subalgebra of the $Q$-algebra $Q_W^*$.
Lemma \ref{lem:bullet_Schuberts} implies that $\NH^*$ is stable under the Hecke action restricted to the nil-Hecke ring $\NH\subset Q_W$.
It is easy to check that the twisted Leibniz rule from Lemma~\ref{lem:bullet_properties}
translates to the following analogous property on Hecke action
\begin{equation}\label{eqn:Leibniz_Hecke}
Y_i\bullet (f\cdot g)=(Y_i\bullet f)\cdot g + (\de_{s_i}\bullet f)\cdot(Y_i\bullet g)
\end{equation}
Using the twisted Leibniz rule, we can prove the following formula for the structure coefficients $p_{u,v}^w$.

\begin{thm}\label{thm:Hecke_product_formula}
Let $w\in W$ and fix a reduced word $w=s_{i_1}\cdots s_{i_r}$, then for any $u,v\in W$ we have
\[
p_{u,v}^w=\sum_{(j_1,\ldots, j_s)} (Y_{i_1}\cdots \hat{Y}_{j_1}\cdots \hat{Y}_{j_s}\cdots Y_{i_r})\bullet \xi_v(\de_e)
\]
where $\hat{Y}_{j}$ means replace $Y_j$ with $\de_{s_j}$ and
the sum is over all subsequences $(j_1,\ldots, j_s)$ of $(i_1,\ldots,i_r)$
for which $s_{j_1}\cdots s_{j_r}$ is a reduced word of $u$.
\end{thm}

\begin{proof}
First note that
\[
Y_w\bullet (\xi_u\cdot \xi_v)=\sum_{w'\in W}p_{u,v}^{w'}\, (Y_w\bullet\xi_{w'}).
\]
Since $Y_e=\de_e$, Lemma \ref{lem:bullet_Schuberts} implies that
\[
p_{u,v}^{w}=Y_w\bullet (\xi_u\cdot \xi_v)(\de_e).
\]
Write $Y_w=Y_{i_1}\cdots Y_{i_r}$.
Lemma \ref{lem:bullet_Schuberts} together with repeated applications of the twisted Leibniz rule in formula~\eqref{eqn:Leibniz_Hecke} gives
\begin{align*}
Y_w\bullet (\xi_u\cdot \xi_v)&=\sum_{(j_1,\ldots, j_s)} ((Y_{j_1}\cdots Y_{j_s})\bullet\xi_u)\cdot((Y_{i_1}\cdots \hat{Y}_{j_1}\cdots \hat{Y}_{j_s}\cdots Y_{i_r})\bullet \xi_v)\\
&=\sum_{(j_1,\ldots, j_s)} \xi_e\cdot((Y_{i_1}\cdots \hat{Y}_{j_1}\cdots \hat{Y}_{j_s}\cdots Y_{i_r})\bullet \xi_v),
\end{align*}
where the sum is over all subsequences $(j_1,\ldots, j_s)$ of $(i_1,\ldots,i_r)$ for which $s_{j_1}\cdots s_{j_r}$ is a reduced word of $u$.
Evaluating both sides at $\de_e$ gives the desired result.
\end{proof}
In the next example, we compute some coefficients $p_{u,v}^w$ using Theorem \ref{thm:Hecke_product_formula} and compare them to the recursive formula given in Theorem \ref{thm:coproduct_recursion}.
\begin{ex}
We consider some examples of type $\tA_2$.  As in Example \ref{ex:coproduct_recursion}, we use $``i"$ to denote $s_i$.
We use Lemmas \ref{lem:bullet_Schuberts} and \ref{lem:bullet_group_basis} to calculate the Hecke action of $Y_i$ and $\de_{s_i}$ respectively. First we have
\[
p_{1,12}^{121}=(\de_1Y_{21}\bullet \xi_{12})(\de_e)+ (Y_{12}\,\de_1\bullet\xi_{12})(\de_e)=0+ (Y_{12}\bullet\xi_{12})(\de_e)=\xi_e(\de_e)=1.
\]
Next, we have
\begin{align*}
p_{1,21}^{121}&=(\de_1Y_{21}\bullet \xi_{21})(\de_e)+ (Y_{12}\,\de_1\bullet\xi_{21})(\de_e)\\
&=1 + Y_{12}\bullet\big(\xi_{21}+s_2(\al_1)\xi_2-\al_1^{\vee}(\al_1)\xi_{21}-\al_1^{\vee}(s_2(\al_1))\xi_{12}\big)(\de_e)\\
&=1 + Y_{12}\bullet(s_2(\al_1)\xi_2-\xi_{21}+\xi_{12})(\de_e)\\
&=1-\xi_e(\de_e)=0.
\end{align*}
We contrast this calculation with other formulas for $p_{u,v}^w$.  Using the recursive formula from Theorem \ref{thm:coproduct_recursion}, we get
$$p_{1,12}^{121}=Y_1(p_{1,12}^{21})+s_1(p_{e,12}^{21}+p_{1,2}^{21})+\al_1s_1(p_{e,2}^{21})=s_1s_2(p_{1,e}^1)=1$$
and
$$p_{1,21}^{121}=Y_1(p_{1,21}^{21})+ s_1(p_{e,21}^{21})=Y_1(s_2(\al_1))+ 1=(1-2)+1=0.$$
Using $\ii$-admissible bounded bijections from Theorem \ref{thm:bbijections2} with $\ii=(1,2,1)$, we get
$$p_{1,12}^{121}=p^\ii_{\{3\},\{1,2\}}+p^\ii_{\{1\},\{1,2\}}=1+0=1$$
and
$$p_{1,21}^{121}=p^\ii_{\{3\},\{2,3\}}+p^\ii_{\{1\},\{2,3\}}=-1+1=0.$$
\end{ex}

In the case of $\ell(u)=1$, we can use Proposition \ref{prop:local_coefficients_prop2} to recover Chevalley's formula.

\begin{thm}
For any $v, s_i\in W$, we have
\[\xi_{s_i}\cdot\xi_v=d_{s_i,v}\cdot\xi_v+\sum_{v\xra{\be} w} \be^{\vee}(\om_i)\cdot\xi_w.\]
\end{thm}

\begin{proof}
Using Proposition~\ref{prop:local_coefficients_prop2} and formula~\eqref{eqn:Hecke_action_group_basis}, we have that
\begin{align*}
\xi_{s_i}\cdot\xi_v=\Big(\sum_{w\in W} d_{s_i,w} \psi_w\Big)\cdot\xi_v &=\sum_{w\in W} (\om_i-w(\om_i)) \psi_w\cdot\xi_v\\
&= \sum_{w\in W}(\om_i\xi_v)\cdot\psi_w-\sum_{w\in W}((\om_i\de_e)\bullet \psi_w)\cdot\xi_v\\
&= \Big(\sum_{w\in W} \psi_w\Big)(\om_i\xi_v-(\om_i\de_e)\bullet \xi_v)\\
&= \om_i\xi_v-(\om_i\de_e)\bullet \xi_v.
\end{align*}
Formula~\eqref{eqn:bullet_group_action1} from Lemma \ref{lem:bullet_group_basis} says that
\[(\om_i\de_e)\bullet \xi_v=v(\om_i)\xi_v-\sum_{v\xra{\be} w}\be^{\vee}(\om_i) \xi_w.\]
Combining these two formulas proves the Chevalley formula.
\end{proof}

Next, we give a formula for $p_{u,v}^w$ in terms of the transition coefficients. As a consequence, we an analogue of the formula given in~\eqref{eqn:cuv_formula} for the coefficients $d_{u,v}$.  This formula is sometimes called the Billey localization formula and was first described in \cite{Bi99}.  For a survey on Billey's formula and its applications, see \cite{Ty}.

\begin{thm}
For any $u,v\in W$, let
\[\xi_u\cdot \xi_v=\sum_{w\in W}p_{u,v}^w\, \xi_w.\]
Then
\[p_{u,v}^w=\sum_{x\in W}d_{u,x}\cdot d_{v,x}\cdot c_{w,x}.\]
In particular, $p_{u,v}^v=d_{u,v}$ and for any reduced word  $v=s_{i_1}\cdots s_{i_r}$, we have
\begin{equation}\label{eqn:Billey_formula}
d_{u,v}=\sum_{(j_1,\ldots, j_s)} \be_{j_1}\cdots \be_{j_s}
\end{equation}
where the sum is over all subsequences $(j_1,\ldots, j_s)\subseteq(i_1,\ldots,i_r)$ such that $s_{j_1}\cdots s_{j_s}$ is a reduced word for $u$ and $\be_k:=s_{i_1}\cdots s_{i_{k-1}}(\al_{i_k})$.
\end{thm}

\begin{proof}
By Lemma \ref{lem:group-schubert_product}, we have
\begin{align*}
\xi_u\cdot\xi_v=\sum_{x\in W} d_{u,x}(\psi_{x}\cdot\xi_v)&=\sum_{x\in W}d_{u,x} \sum_{w\in W}(c_{w,x}\cdot d_{v,x})\, \xi_{w}\\
&=\sum_{w\in W}\Big(\sum_{x\in W}d_{u,x}\cdot c_{w,x}\cdot d_{v,x}\Big)\, \xi_{w}.
\end{align*}
Now if $d_{u,x}\cdot c_{w,x}\cdot d_{v,x}\neq 0$, then Proposition \ref{prop:transistion_coefficients} part (1) implies that $v\leq x\leq w$ (note the we also must have $u\leq x\leq w$).  Hence if $w=v$, then \[p_{u,v}^v=d_{u,v}\cdot (c_{v,v}\cdot d_{v,v})=d_{u,v}\]
by Proposition \ref{prop:transistion_coefficients} part (2).  Finally, applying the recursive formula from Theorem \ref{thm:coproduct_recursion} to the coefficient $p_{u,v}^v$ shows that it satisfies formula~\eqref{eqn:Billey_formula} and hence so does $d_{u,v}$.
\end{proof}


\section{Equivariant cohomology of the flag variety}\label{sec:eqcoh}

In this section, we outline the classical connections between torus-equivariant singular cohomology of flag varieties and nil-Hecke rings following \cite[Chapter 11.3]{Ku02}.
For more details on the definition and properties of equivariant cohomology of flag varieties see \cite{Anderson-Fulton-book} and \cite{CZZ2}.

\subsection{\it The localization map}
Let $G$ be a split semi-simple linear algebraic group over a field $\kk$ of characteristic $0$ together with
a split maximal torus $T$ and a Borel subgroup $B$ containing $T$.
The (complete) flag variety of $G$ is defined to be the quotient space $G/B$.
It is a smooth projective algebraic variety with a left $T$-action given by the usual group product.
The Weyl group $W:=N_G(T)/T$ of $G$ indexes the $B$-orbits of the flag variety in the Bruhat decomposition
\[
G/B=\bigsqcup_{w\in W} BwB/B.
\]
The $B$-orbits $BwB/B$ are called the Schubert cells and their closures $\overline{BwB/B}$ are called the Schubert varieties.

Consider the $T$-equivariant singular cohomology ring $\SC_T^*(G/B;\zz)$ of $G/B$ with integer coefficients.
Its odd-degree part vanishes (since $G/B$ is a cellular space),
and its even-degree part can be identified with the equivariant Chow ring $\CH_T^*(G/B)$ (see \cite{Br}).
Via this identification the coefficient ring gives the symmetric algebra:
\[
\SC_T^{2*}(pt;\zz)\simeq S=\Sym_{\zz}(\cl),
\]
where $\cl$ is the character group of $T$.  For instance, if $G$ is simply-connected,
then $\cl$ is the weight lattice and $\SC_T^{2*}(pt;\zz)\simeq S$ (the isomorphism of graded rings)
is the polynomial ring in fundamental weights.
In what follows, we will always deal with the even-degree part $\SC^{2i}_T(G/B;\zz)$
which we denote by $\CH_T^i(G/B)$.

For each Schubert variety, we define the $T$-equivariant Schubert class
\[
\sigma^T_w:=[\overline{BwB/B}]\in \CH_T^*(G/B).
\]
It is well known that as $S$-modules:
\[
\CH_T^*(G/B)\simeq S\otimes_{\zz}\left( \bigoplus_{w\in W} \zz\cdot \sigma^T_w\right).
\]
Hence, the Schubert classes form a basis of $\CH_T^*(G/B)$ as a $S$-module.
Each $B$-orbit contains a unique $T$-fixed point $wB/B$ and, therefore, the Weyl group also indexes the fixed point set
\[
(G/B)^T=\bigsqcup_{w\in W}\{wB/B\}.
\]
The inclusion
\[
\imath\colon (G/B)^T\hra G/B
\]
induces an injective ring homomorphism on $T$-equviarant cohomology
\[
\imath^*\colon \CH_T^*(G/B)\to \CH_T^*((G/B)^T)\simeq \bigoplus_{w\in W}S.
\]
The induced map $i^*$ is called the localization (moment) map on $T$-equivariant cohomology.
Recall that the dual to the nil-Hecke ring $\NH^*$ is the $S$-submodule of $Q^*_W$ with Schubert basis $\{\xi_w\}_{w\in W}$.
We define a $S$-linear map
\[
\nu\colon \NH^*\to \bigoplus_{w\in W}S
\]
by mapping each Schubert basis element
\[
\nu(\xi_w):=\bigoplus_{v\in W}\xi_w(\de_v)=\bigoplus_{v\in W}d_{w,v}.
\]
Proposition \ref{prop:local_coefficients_prop2} and Lemma \ref{lem:semisimple_lattices} imply that
each $d_{w,v}\in S$ and hence $\nu$ is well-defined.
In fact, it is not difficult to show that $\nu$ is an injective $S$-algebra homomorphism.
Thus, the image of $\nu$ identifies the dual of the nil-Hecke ring $\NH^*$ as a $S$-subalgebra of $\CH_T^*((G/B)^T).$

\begin{thm}\label{thm:equiv_cohomology}
The image $\image(\imath^*)=\nu(\NH^*)\subseteq \CH_T^*((G/B)^T)$, and thus there is a $S$-algebra isomorphism
\[
\imath^*\colon \CH_T^*(G/B)\to\nu(\NH^*)\simeq \NH^*
\]
with $\imath^*(\sigma_w^T)=\nu(\xi_w)$ for any $w\in W$.
In particular, the $T$-equviariant Schubert structure constants $p_{u,v}^w$ defined by the cup product
\[
\sigma_v^T\cdot \sigma_u^T=\sum_{w\in W}p_{u,v}^w\, \sigma_w^T
\]
are the same as the structure constants given in~\eqref{eqn:Schubert_structure_constants}.
\end{thm}


\subsection{\it The evaluation map}
The base-point inclusion map induces a canonical evaluation map
\[
\eta\colon \CH_T^*(G/B)\to \CH^*(G/B).
\]
In particular, the map $\eta$ induces a graded $\zz$-algebra isomorphism
\[
\bar\eta\colon \zz\otimes_{S} \CH_T^*(G/B)\to \CH^*(G/B),
\]
where $\bar\eta(1\otimes p):=\eta(p)$.
Recall from Section \ref{S:augmented_coproduct} that the augmentation map $\veps\colon S\to R=\zz$ induces a map
\[
\veps_*\colon \NH \to \veps\NH\simeq \zz\otimes_{S}\NH.
\]
Augmentation also gives rise to an analgous map on the dual of the nil-Hecke ring
\[
\veps^*\colon \NH^*\to \veps\NH^*:=\zz\otimes_{S}\NH^*
\]
such that dualizing the co-commutative coproduct $\DI^{\veps}$ on $\veps\NH$
gives a commutative product structure on $\veps\NH^*$.
The $\zz$-algebra $\veps\NH^*$ and a Schubert basis $\{\veps_w\}_{w\in W}$
where $\veps_w:=\veps^*(\xi_w)$ and
\[
\veps_u\cdot \veps_v=\sum_{w\in W} c_{u,v}^w\, \veps_w.
\]
Define the map
\[
\bar{\imath}^*\colon \CH^*(G/B)\to \veps\NH^*
\]
by $\bar{\imath}^*(p):=\imath^*(\bar p)$ where $1\otimes \bar p:=\bar\eta^{-1}(p)$.

\begin{thm}\label{thm:cohom_eval}
The map $\bar{\imath}^*\colon \CH^*(G/B)\to \veps\NH^*$ is a $\zz$-algebra isomorphism such that the following diagram commutes:
\[
\xymatrix{
\CH^*_T(G/B) \ar[r]^{\eta} \ar[d]_{\tilde \imath^*}& \CH^*(G/B)  \ar[d]^{\bar{\imath}^*} \\
\NH^* \ar[r]^{\veps^*}& \veps\NH^*
}
\]
Specifically, if we denote the Schubert classes in $\CH^*(G/B)$ by $\sigma_w:=\eta(\sigma_w^T)$ , then $\bar{\imath}^*(\sigma_w)=\veps_w$ and
\[
\sigma_u\cdot \sigma_v=\sum_{w\in W} c_{u,v}^w\, \sigma_w.
\]
\end{thm}


\subsection{\it Partial flag varieties}
Let $\mathcal{R}$ denote the simple generating set of the Weyl group $W$ corresponding to the Borel subgroup $B$.
For any subset $J\subseteq \mathcal{R}$, let $W_J$ denote the subgroup generated by $J$ and
let \[P_J:=BW_JB\] denote the standard parabolic subgroup of $J$.
As with the complete flag variety, the partial flag variety, defined as the quotient space $G/P_J$,
is also a smooth, projective algebraic variety with a left $T$-action.
The $T$-fixed points and Schubert subvarieties of $G/P_J$ are indexed
by the set of minimal length coset representatives $W^J\simeq W/W_J$.
For any $w\in W^J$, we have the Schubert classes
\[
\sigma_{w,J}^T:=[\overline{BwP/P}]\in \CH_T^*(G/P_J)
\]
and
\[
\CH_T^*(G/P_J)\simeq S \otimes_{\zz}\left( \bigoplus_{w\in W^J} \zz\cdot \sigma^T_{w,J}\right).
\]
Since $B\subseteq P_J$, there is a natural projection map
\[
\pi\colon G/B\twoheadrightarrow G/P_J
\] given by $\pi(gB)=gP$ which induces an injective ring homomorphism
\[
\pi^*\colon \CH^*_T(G/P_J)\to \CH^*_T(G/B).
\]
On Schubert classes, we have for any $w\in W^J$, that $\pi^*(\sigma_{w,J}^T)=\sigma_w^T$.
The ring $\CH_T(G/P_J)$ can be realized as a certain sub-ring of invariants in $\NH^*$.
Recall the Hecke action of $Q_W$ on $Q_W^*$ given in Definition \ref{Def:Hecke_action}.
Lemma \ref{lem:bullet_group_basis} implies that twisted group algebra $S_W$
preserves the dual of the nil-Hecke ring $\NH^*\subseteq Q_W^*$.
Hence the Hecke action gives rise to an action of the Weyl group $W$ on $\NH^*$ given by
\[
w\bullet \xi_v:=(\de_w)\bullet \xi_v.
\]
Let $(\NH^*)^{J}$ denote the sub-algebra of $W_J$-invariants in $\NH^*$ and let
\[
\tilde\pi^*\colon (\NH^*)^{J}\to \NH^*
\]
denote the inclusion of algebras.
Lemma \ref{lem:bullet_group_basis} implies $\xi_w\in (\NH^*)^{J}$ if and only if $w\in W^J$.
Consider the composition of maps
\[
\imath^*\circ\pi^*\colon \CH_T(G/P_J)\to \CH_T^*((G/B)^T).
\]
We denote this composition by $\tilde\imath^*:=i^*\circ\pi^*$.

\begin{thm}
The map $\tilde \imath^*\colon \CH^*_T(G/P_J)\to \CH_T^*((G/B)^T)$ induces an $S$-algebra isomorphism
\[
\tilde \imath^*\colon \CH^*_T(G/P_J)\to \nu((\NH^*)^{J})\simeq (\NH^*)^{J}
\]
such that the following diagram commutes:
\[
\xymatrix{
\CH^*_T(G/P_J) \ar[r]^{\pi^*} \ar[d]^{\tilde \imath^*}& \CH^*_T(G/B)  \ar[d]^{\imath^*} \\
(\NH^*)^{J} \ar[r]^{\tilde\pi^*}& \NH^*
}
\]
In particular, for any $w\in W^J$, we have $\tilde \imath^*(\sigma_{w,J}^T)=\xi_w$ and hence
\[
\sigma_{u,J}^T\cdot \sigma_{v,J}^T=\sum_{w\in W^J}p_{u,v}^w\, \sigma_{w,J}^T.
\]
\end{thm}

As with the full flag variety, the partial flag variety also has an evaluation map
\[
\eta\colon \CH_T^*(G/P_J)\to \CH^*(G/P_J).
\]
We also have the corresponding subalgebra of $W_J$-invariants on the dual to the augmented nil-Hecke ring
\[
(\veps\NH^*)^{J}=\zz\otimes_{S} (\NH^*)^{J}.
\]
The following theorem is a generalization of Theorem \ref{thm:cohom_eval} for partial flag varieties

\begin{thm}
The map $\bar \imath^*\colon \CH^*(G/P_J)\to (\veps\NH^*)^{W_{J}}$ is a $\zz$-algebra isomorphism
such that the following diagram commutes:
\[
\xymatrix{
\CH^*_T(G/P_J) \ar[r]^{\eta} \ar[d]_{\tilde \imath^*}& \CH^*(G/P_J)  \ar[d]^{\bar \imath^*} \\
(\NH^*)^{J} \ar[r]^{\veps^*}& (\veps\NH^*)^{J}
}
\]
If we denote the Schubert classes in $\CH^*(G/P_J)$ by $\sigma_{w,J}:=\eta(\sigma_{w,J}^T)$ for $w\in W^J$,
then $\bar{\imath}^*(\sigma_{w,J})=\veps_w$ and
\[
\sigma_{u,J}\cdot \sigma_{v,J}=\sum_{w\in W^P} c_{u,v}^w\, \sigma_{w,J}.
\]
\end{thm}

\begin{rem}  It has been classically known for a long time that the augmented Schubert structure constants $c_{u,v}^w$ are non-negative integers.  The proof is geometric in nature and uses the intersection theory of flag varieties.  In \cite{Ga01}, Graham proved, using invariant cycles, that the $T$-equivariant Schubert structure constants $p_{u,v}^w$ are non-negative in $S(\cl_r)\subseteq S(\cl)$.  In other words, he shows that $p_{u,v}^w$, as a polynomial in the simple roots, has non-negative integer coefficients.  Later in \cite{An07}, Anderson gave an alternate proof of this fact using intersection theory which is analogous to the classical geometric proof for augmented coefficients.  It is an open question in algebraic combinatorics to find a combinatorial explanation for the positivity properties of the coefficients $c_{u,v}^w$ and $p_{u,v}^w$.  While the nil-Hecke ring provides many algebraic formulas for these coefficients, none have been shown to be combinatorially or algebraically positive outside a few exceptional cases.
\end{rem}

\subsection{\it Smoothness criterion}
We finish this section with a smoothness criteria for Schubert varieties derived from the nil-Hecke ring.  This formula is proved by Kumar in \cite{Ku96}.  By definition the Schubert variety is the closure of a $B$-orbit in $G/B$ and its boundary is given by the Bruhat partial order on $W$.
\[X(w):=\overline{BwB/B} =\bigsqcup_{v\leq w}BvB/B\]

In general, Schubert varieties are singular with its singular locus comprising of a union of $B$-orbits on the boundary of their open Schubert cell.  To determine if a $B$-orbit is singular, it suffices to check if the corresponding $T$-fixed point within the orbit is singular.  The singularity/smoothness of such points can be checked using the nil-Hecke ring.  In fact, the nil-Hecke ring can also detect when a $T$-fixed point (and hence, a $B$-orbit) is rationally smooth.  Recall the change of basis coefficients $c_{w,v}\in Q$ defined in~\eqref{eqn:c_wv_coeficients}
and for $v\leq w$, define the set $$S(w,v):=\{\be\in \RS^+ \ |\ s_\be v\leq w\}.$$
The next proposition is proved in \cite{Ku96}:

\begin{prop}
  For $v\leq w$, there exists a unique homogeneous polynomial $c'_{w,v}\in S$ of degree $|S(w,v)|-\ell(w)$ such that
  $$c_{w,v}= c'_{w,v}\cdot (-1)^{\ell(w)-\ell(v)} \cdot\prod_{\be\in S(w,v)} \be^{-1}.$$
\end{prop}

By Proposition \ref{prop:transistion_coefficients}, if $v=w$, then $S(w,w)=\RS^+\cap w\RS^-$ and $c'_{w,w}=1$.  The following theorem states Kumar's criterion for smoothness of a $T$-fixed point of a Schubert variety.


\begin{thm}\label{thm:smoothness}
    The following are true:
    \begin{enumerate}
        \item[(i)] The point $vB$ is a smooth in $X(w)$ if and only if $c'_{w,v}=1$.
        \item[(ii)] The point $vB$ is a rationally smooth in $X(w)$ if and only if $c'_{w,y}$ is a positive integer for all $v\leq y\leq w$.
    \end{enumerate}
\end{thm}

We remark that if $c'_{w,y}$ is a constant, then it can be shown that it is in fact a positive integer.  Since the singular locus of $X(w)$ is closed, to check for global smoothness of $X(w)$, it suffices to check smoothness only at the point $eB$ (note the same is not true for rational smoothness).

Kumar's criterion readily extends to Schubert varieties of partial flags.  Let $J\subseteq \mathcal{R}$ denote a subset of simple generators and let $W_J\subseteq W$ denote the Weyl subgroup generated by $J$.  For any $v\in W^J$, let
\[X^J(v):=\overline{BvP/P}\]
denote the corresponding Schubert variety.  The preimage of the map $\pi\colon G/B\to G/P_J$ restricted to the $X^J(v)$ is
\[\pi^{-1}(X^J(v))=X(vu_0)\] where $u_0$ denotes the maximal element of $W_J$.  Furthermore the projection $\pi$ is a $P_J/B$-fiber bundle on $G/B$ and hence, restricts to $P_J/B$-fiber bundle on $X(vu_0)$ with base $X^J(v)$.  Since the fiber $P_J/B$ is smooth, we get the following:

\begin{cor}
Let $v\in W^J$.  The following are equivalent:
\begin{enumerate}
\item[(i)] The Schubert variety $X^J(v)$ is smooth.
\item[(ii)] The Schubert variety $X(vu_0)$ is smooth.
\item[(iii)] The coefficient $c'_{vu_0,e}=1$.
\end{enumerate}
\end{cor}

\begin{ex}\label{ex:smooth}
Consider the root system of type $\tA_3$ and let $w=s_2s_1s_3s_2$. It is well known that the Schubert variety $X(w)$ has a singular point at $\{eB\}$.  We verify this fact using Theorem \ref{thm:smoothness}.  First, we calculate the coefficient $c_{w,e}$ using Proposition \ref{prop:transistion_coefficients}.  There are two subsequences of the reduced word $(s_2,s_1,s_3,s_2)$ corresponding to the identity $e$: the subsequence $(s_2,s_2)$ and the empty sequence.   Equation \eqref{eqn:c_wv_coeficients} from Proposition \ref{prop:transistion_coefficients} implies

\begin{align*}
c_{w,e}&=\frac{1}{\al_2\,s_2(\al_1)\,s_2(\al_3)\,s_2(\al_2)}+\frac{1}{\al_2\,\al_1\,\al_3\,\al_2}=\frac{-\al_1\al_3+s_2(\al_1)s_2(\al_3)}{\al_2^2\,\al_1\,\al_3\,s_2(\al_1)\,s_2(\al_3)}\\
&=\frac{\al_1+\al_2+\al_3}{\al_1\,\al_2\,\al_3\,s_2(\al_1)\,s_2(\al_3)}
\end{align*}
The set $S(w,e)=\{\al_1, \al_2, \al_3, s_2(\al_1), s_2(\al_3)\}$ and hence
\[c'_{w,e}=\al_1+\al_2+\al_3.\]  Theorem \ref{thm:smoothness} implies that $eB$ is a singular point of $X(w)$.  A similar calculation confirms that $c'_{w,v}=1$ for all $e< v\leq w.$  For examples in rank 2 of Kumar's smoothness criterion see \cite[Lemma 6.2]{Ku96}.
\end{ex}


\section{Equivariant connective K-theory of the flag variety}\label{sec:eqkth}

In this section we demonstrate how the results of previous sections can be extended to the context of the so-called connective $K$-theory.
Our main sources here are the paper \cite{KK90} by Kostant-Kumar on the $K$-theory, papers \cite{CZZ,CZZ1,CZZ2} on generalized oriented cohomology theories and \cite{HM20,KM19} on the equivariant connective $K$-theory which we mostly follow.

\subsection{\it Algebraic oriented cohomology and formal group laws} We first turn to generalized cohomology theories and formal group laws.

Let $\hh$ be a graded algebraic {\em oriented cohomology} theory (AOCT)  introduced by Levine and Morel in~\cite{LM}, i.e.~$\hh$ is a contravariant functor from the category of smooth quasi-projective varieties over a field $\kk$ of characteristic $0$ to the category of graded commutative rings which satisfies the axioms of \cite[Def.1.1.2]{LM}.
The word `oriented' here reflects the fact that there is a 1-1 correspondence between the Euler structures (characteristic classes) on $\hh$:
\begin{itemize}
\item
To each line bundle $\Ll$ over a variety $X$ one assigns the characteristic (Euler) class $c_1^{\hh}(\Ll):=s^*s_*(1_X) \in \hh^1(X)$, where $s$ is the zero-section of $\Ll$ and $s^*$ (resp. $s_*$) is the induced pull-back (resp. push-forward).
\end{itemize}
and the choices of a local parameter $x^{\hh}$ (orientations):
\begin{itemize}
\item
Choose $x^{\hh}=c_1^{\hh}(\mathcal{O}(1)) \in \hh^1(\mathbb{P}^\infty)=R[[x]]^{(1)}$, where $R=\hh(pt)$ is the graded coefficient ring. Observe that $x^{\hh}=a_0x+a_1x^2+\ldots$, where $a_0$ is invertible in $R$ and we assign $\deg a_i:=-i$.
\end{itemize}

\begin{ex}
Basic examples of AOCTs are
\begin{itemize}
\item[(i)]
The Chow theory (modulo rational equivalence)
$\hh(X)=\CH(X)$ over $R=\zz$ \cite[Ex.1.1.4]{LM}.
\item[(ii)]
The graded $K$-theory $\hh(X)=\K(X)[t^{\pm 1}]$ over the Lorentz polynomials $R=\zz[t^{\pm 1}]$ (here we assign $\deg t:=-1$ and $\deg x:=0$ for all $x\in \K(X)$) \cite[Ex.1.1.5]{LM}.
\item[(iii)]
The (graded) connective K-theory $\hh(X)=\CK(X)$ of Cai~\cite{Ca08} over the polynomial ring $R=\zz[t]$.
\end{itemize}
\end{ex}

A key property of any such theory is the Quillen formula \cite[Lemma~1.1.3]{LM}: given two line bundles $\Ll_1$, $\Ll_2$ over $X$ one has
\[
c_1^{\hh}(\Ll_1\otimes \Ll_2)=F(c_1^{\hh}(\Ll_1),c_1^{\hh}(\Ll_2)),
\]
where $F\in R[[x,y]]$ is the associated {\em formal group law}  over the coefficient ring $R$.

Recall  that (e.g. see~\cite{La55}) a formal group law (FGL) $F$ over a commutative unital ring $R$ is a power series in two variables $F(x,y)\in R[[x,y]]$ which is commutative $F(x,y)=F(y,x)$, associative $F(x,F(y,z))=F(F(x,y),z)$ and has neutral element $F(x,0)=x$. By definition we have
\[
F(x,y)=x+y+\sum_{i,j>0} a_{ij}x^i y^j,\;\text{ for some coefficients }a_{ij}\in R.
\]
To respect the grading on $\hh$ we set $\deg a_{ij}=1-i-j$ for the coefficients $a_{ij}$ of $F$.

\begin{ex} For the Chow theory $F(x,y)=x+y$ ($a_{ij}=0$ for all $i,j$).

For both the graded $K$-theory and the connective $K$-theory it is
\[F(x,y)=x+y-txy\] with $\deg t=-1$ ($a_{11}=-t$ and $a_{ij}=0$ for all $i+j>2$).
\end{ex}

The quotient of the polynomial ring in the variables $a_{ij}$s modulo relations imposed by the commutativity and associativity defines the Lazard ring $\LL$.
The respective FGL over $\LL$ denoted $F_U$ is the universal FGL in the following sense: for any FGL $F$ over $R$ there is a map $\LL \to R$ given by evaluating the coefficients $a_{ij}$ of $F_U$.

By the main result of \cite{LM} there is an algebraic oriented cohomology theory $\AC(-)$
called algebraic cobordism which has the associated formal group law $F_U$ over $\LL=\AC(pt)$.
Given a FGL~$F$ over $R$ one may define an AOCT called a free theory by a simple base change of the coefficient rings:
\[
\hh(-):=\AC(-)\otimes_{\LL} R
\]
(here $\LL \to R$ is the universal map introduced above).

\begin{ex}
The Chow theory, the graded $K$-theory and the connective $K$-theory are free theories (see \cite{LM}).
Moreover, the connective $K$-theory is the universal theory for both Chow theory and the $K$-theory, i.e., there is a diagram
of base change maps:
\[
\xymatrix{
& \CK(-) \ar[dl]_{t\mapsto 0} \ar[dr]^{\zz[t]\hra \zz[t^{\pm 1}]} & \\
\CH(-) & & \K(-)[t^{\pm 1}]
}
\]
\end{ex}


\subsection{\it Algebraic connective $K$-theory}
We now recall the geometric definition and basic properties of the connective $K$-theory $\hh(-)=\CK(-)$ following \cite{HM20}.

Given a smooth algebraic variety $X$, for every integer $i$ we denote by $\MM_i(X)$
the abelian category of coherent $\OO_X$-modules with dimension of support at most $i$. So
$\MM_i(X)=0$ if $i<0$
and $\MM_i(X)=\MM(X)$ if $i\ge \dim X$. We then set the connective $K$-groups of $X$ to be
\[
\CK_i(X):=\image\big(\K(\MM_i(X))\to \K(\MM_{i+1}(X))\big),
\]
where $\K(\MM_i(X))$ denotes the Grothendieck group of the respective coherent $\OO_X$-modules.

These groups are related to the Chow groups via the exact sequence
\[
\CK_{i-1}(X) \xra{t} \CK_i(X) \to \CH_i(X) \to 0,
\]
where $t$ is the Bott homomorphism.
We can view the graded group $\CK(X)$ as a module over the polynomial ring $\zz[t]$ and there is an isomorphism of oriented theories
\[
\CK(-)/t\CK(-)\simeq \CH(-).
\]
As for the $K$-theory, there is an isomorphism of oriented theories
\[
\CK(-)/(t-1)\CK(-)\simeq \K(-).
\]

Recall that $\K(\MM(X))=\K(X)$, and there is a topological filtration on $\K(X)$ defined by
\[
\K(X)_{(i)}:=\image\big(\K(\MM_i(X))\to \K(X)\big).
\]
The relation between the groups $\CK_i(X)$, $\CH_i(X)$ and $\K(X)_{(i)}$ is given by the following commutative diagram of surjective maps
\[
\xymatrix{
\CK_i(X) \ar@{>>}[r] \ar@{>>}[d] & \K(X)_{(i)} \ar@{>>}[d] \\
\CH_i(X) \ar@{>>}[r] & \K(X)_{(i)}/\K(X)_{(i-1)}.
}
\]


\subsection{\it Connective group rings}
Let $\cl$ be a finitely generated free abelian group. Let $R=\zz[t]$ be the graded polynomial ring in the variable $t$ with $\deg t=-1$.
Consider the polynomial ring $R[x_\la\colon \la\in \cl]$ in variables $x_\la$
indexed by all non-zero elements of $\cl$ (we assume $x_0=0$). It has the induced grading where $\deg x_\la=1$.

Let $I$ be the ideal generated by all elements
\[
x_{\al+\be}-x_\al-x_\be+t x_\al x_\be,\quad \al,\be\in \cl.
\]
The quotient
\[
\SP(\cl):=R[x_\la\colon \la\in \cl]/I
\]
is the graded ring which we call a {\em connective group ring}. Indeed, this is a particular case of a more general construction of a formal group ring of \cite[\S2]{CPZ}. Namely, the case corresponding to the formal group law $F$ of the connective $K$-theory. For simplicity we write $\SP$ instead of $\SP(\cl)$ having in mind the corresponding lattice.

\begin{rem}
If we set $t:=0$, then $\SP$ turns into the symmetric algebra $S=\Sym_R(\cl)$ of the previous sections (polynomial ring in a basis of $\cl$).
If we set $t=1$, then it turns into the usual group ring $R[\cl]$ of $\cl$ via $x_\la\mapsto 1-e^{-\la}$.
Hence, one may think of $\SP$ as a deformation between the polynomial ring $S$ and the usual group ring $R[\cl]$ of the lattice $\cl$.
\end{rem}

Following~\cite[\S5]{HML} given a free theory $\hh$ with the associated FGL $F$ over the coefficient ring $R=\hh(pt)$,
a split torus $T$ and a smooth $T$-variety $X$ over $\kk$, one may define the associated $T$-equivariant oriented cohomology $\hh_T(X)$
as the limit of usual oriented cohomology taken over a certain system of $T$-representations (the so-called Borel construction).
We refer to \cite{To,EG,Br97} for the definition and properties of the equivariant Chow theory and equivariant $K$-theory;
and to \cite{KM19} for the equivariant connective $K$-theory.
All such $T$-equivariant theories satisfy the axioms of \cite[\S2]{CZZ2} and, moreover,
the  $T$-equivariant coefficient ring $\CK_T(pt)$ of the connective $K$-theory is isomorphic to the connective group ring
$\SP(T^*)$ of the group of characters $T^*$ of $T$:
\[
\SP(T^*)\simeq \CK_T(pt).
\]
Under this identification $x_\la$ corresponds to the first characteristic class $c_1^{\CK}(\Ll_\la)$
of the associated $T$-equivariant line bundle $\Ll_\la$ over $pt$.


We now extend most of the constructions and results of previous sections to the context of the connective $K$-theory and connective group rings following results of \cite{CZZ,CZZ1,CZZ2}.

\subsection{\it The connective Hecke ring}
As before let $G$ be a split semisimple linear algebraic group $G$ over a field $\kk$ of characteristic $0$ together with a a split maximal torus $T$ and a Borel subgroup $B$ containing $T$. Consider the associated finite crystallographic (semisimple) root datum $\RS \hra \cl^\vee$ together with its subset of simple roots $\SR$ and the character lattice $\cl:=T^*$. The Weyl group $W$ of $\RS$ acts on $T^*$ and, hence, on the connective group ring $\SP=\SP(\cl)$. It is well-known that this action coincides with the natural action of $W$ on the coefficient ring $\CK_T(pt)$.

Similarly to the definition of the twisted group algebra for $S$
we define the twisted group algebra for $\SP$ to be the free left $\SP$-module \[\SP_W:=\SP \otimes_R R[W]\] so that each element of $\mathcal{S}_W$ can be written uniquely as a linear combination $\sum_{w\in W} p_w\de_w$, where $\{\de_w\}_{w\in W}$ is the standard basis of the group ring $R[W]$ and $p_w\in \SP$.
It has the product induced  by that of $R[W]$ and the twisted commuting relation $w(p)\de_w =\de_w p$,  $p\in \SP$.

\begin{dfn}
We consider the localization $\mathcal{Q}:=\SP(\tfrac{1}{x_\al}, \;\al\in \RS)$ and define the $R$-subalgebra $\DF$ of
$\mathcal{Q}_W$ to be generated by the divided-difference elements corresponding to simple reflections
\[X_i:=\tfrac{1}{x_{\al_i}}(\de_e-\de_{s_i}),\quad i=1\ldots n,\] and multiplications by elements from $\SP$. We call it the {\em connective Hecke ring}.
\end{dfn}

The connective Hecke ring satisfies the same braid relations as the nil-Hecke ring (see e.g. \cite{BE90}).
But instead of the nilpotent relation it satisfies the idempotent-type relation
\begin{align*}
X_i^2 &=\big(\tfrac{1}{x_{\al_i}}-\tfrac{1}{x_{\al_i}}\de_{s_i}\big)\big(\tfrac{1}{x_{\al_i}}-\tfrac{1}{x_{\al_i}}\de_{s_i}\big)\\
 &=\big(\tfrac{1}{x_{-\al_i}^2}+\tfrac{1}{x_{\al_i}x_{-\al_i}}\big)-\big(\tfrac{1}{x_{\al_i}^2}+\tfrac{1}{x_{\al_i}x_{-\al_i}}\big)\de_i\\
 &=\big(\tfrac{1}{x_{\al_i}}+\tfrac{1}{x_{-\al_i}})X_i.
\end{align*}
Since in the connective group ring we have
\[
0=x_{\al_i+(-\al_i)}=x_{\al_i}+x_{-\al_i}-tx_{\al_i}x_{-\al_i},
\]
we obtain
\[
\tfrac{1}{x_{\al_i}}+\tfrac{1}{x_{-\al_i}}=t \in \mathcal{Q}.
\]
So the idempotent relation then becomes
\[
X_i^2=tX_i.
\]
As for the affine relation (Aff), we get
\begin{align*}
X_i x_\la &=\tfrac{1}{x_{\al_i}}(\de_e-\de_{s_i})x_\la=\tfrac{x_\la}{x_{\al_i}}\de_e-\tfrac{s_i(x_\la)}{x_{\al_i}}\de_{\al_i}\\
 &= s_i(x_\la)\tfrac{1}{x_{\al_i}}(\de_e-\de_{\al_i})+\tfrac{x_\la}{x_{\al_i}}-\tfrac{s_i(x_\la)}{x_{\al_i}}\\
  &=s_i(x_\la) X_i+X_i\circ x_\la,
\end{align*}
where $X_i\circ x_\la:=\tfrac{x_\la-s_i(x_\la)}{x_{\al_i}}$.

Set $k:=\al_i^\vee(\la)$. We then obtain
\begin{align*}
s_i(x_\la) &= x_{s_i(\la)}=x_{\la- \al_i^\vee(\la)\al_i}\\
 &= x_\la + x_{-k\al_i} - tx_\la x_{-k\al_i}\\
 &=x_\la -x_{-k\al_i}(tx_\la-1).
\end{align*}
Therefore, in $\SP$ we get
\[
X_i\circ x_\la=\tfrac{x_{-k\al_i}}{x_{\al_i}}(tx_\la -1).
\]
Observe that in the connective group ring $\SP$ for any integer $k$ and $\be\in \cl$ we have
\[
1-t x_{k\be}=(1-tx_\be)^k.
\]
This allows to compute the quotient $\tfrac{x_{-k\al_i}}{x_{\al_i}}$ which will be again in $\SP$.
\begin{ex}
If $t=1$ (the usual $K$-theory case), then in the group ring $\zz[\cl]$ using the exponential notation ($x_\la\mapsto 1-e^{-\la}$) we get
\[
\tfrac{x_{-k\al_i}}{x_{\al_i}}=\tfrac{1-e^{k\al_i}}{1-e^{-\al_i}}=-e^{\al_i}\tfrac{1-e^{k\al_i}}{1-e^{\al_i}},
\]
which gives \[
X_i\circ x_\la=\tfrac{1-e^{k\al_i}}{1-e^{\al_i}}e^{\al_i-\la} .
\]
\end{ex}

\begin{rem}
The algebra $\DF$ can be viewed as a generalization of the nil-Hecke ring $\NH$ (as the latter can be obtained by setting $t:=0$ in $\SP$).
Moreover, if $t=1$, then it gives the classical 0-affine Hecke algebra.
\end{rem}

\subsection{\it The coproduct}
Following arguments of section~\ref{sec:coprod} (see \cite{CZZ}) one constructs the cocommutative coproduct $\DI$ on $\DF$ so that $\DI$ turns the $\SP$-linear dual $\DF^*:=\Hom_{\SP}(\DF,\SP)$ into a commutative $\SP$-algebra.

Given $w\in W$ and its reduced expression
$w=s_{i_1}s_{i_2}\ldots s_{i_r}$
we set \[X_w:=X_{i_1}\ldots X_{i_r}\in \mathcal{Q}_W.\]
We then obtain the following formulas (cf Theorem~\ref{thm:coproduct_recursion}) for the coproduct $\DI(X_w)$ in terms of $X_i$s:
\begin{thm}
For any $i\leq n$, we have
\[
\DI(X_i)=X_i\otimes 1+1\otimes X_i -x_{\al_i} X_i\otimes X_i,
\]
For any $w\in W$ and $s_i\in D_{L}(w)$,  if we write
\[
\DI(X_w)=\sum_{v,u\in W} p_{u,v}^w \ X_u\otimes X_v,
\]
then the coefficients satisfy the recursion
\[
p_{u,v}^w=\begin{cases}
X_i\circ(p_{u,v}^{s_iw})  & \text{if $s_i\notin D_L(u)\cup D_L(v)$}\\
X_i\circ(p_{u,v}^{s_iw}) +s_i(p_{s_iu,v}^{s_iw}+tp_{u,v}^{s_iw})&\text{if $s_i\in D_L(u)\setminus D_L(v)$}\\
X_i\circ(p_{u,v}^{s_iw}) +s_i(p_{u,s_iv}^{s_iw}+tp_{u,v}^{s_iw})&\text{if $s_i\in D_L(v)\setminus D_L(u)$}\\
X_i\circ(p_{u,v}^{s_iw}) +s_i(p_{s_iu,v}^{s_iw}+p_{u,s_iv}^{s_iw}+2tp_{u,v}^{s_iw}) & \\
\qquad\qquad -x_{\al_i}s_i(p_{s_iu,s_iv}^{s_iw}+t^2p_{u,v}^{s_iw})
&\text{if $s_i\in D_L(u)\cap D_L(v)$}
\end{cases}
\]
with $p_{e,e}^e=1$.
\end{thm}

\begin{proof} We only prove the formulas for the coefficients.
As in the cohomology case consider $w\in W$ with $\ell(w)\geq 2$, and let $w=s_iw'$ from some $s_i\in D_L(w)$.
Write
\[
\DI(X_{w'})=\sum_{u,v} p_{u,v}^{w'}\ X_u\otimes X_v
\]
and assume for the sake of induction that $\DI(X_{w'})\in \DF\otimes_{\SP} \DF$, or equivalently, each $p_{u,v}^{w'}\in \SP$.
Then
\begin{align*}
\DI(X_w)&=\DI(X_i)\odot\DI(X_{w'})\\
&=\tfrac{1}{x_{\al_i}}(\de_e\otimes\de_e - \de_{s_i}\otimes\de_{s_i})\odot\big(\sum_{u,v} p_{u,v}^{w'}\ X_u\otimes X_v\big)\\
&= \sum_{u,v}\ \tfrac{1}{x_{\al_i}}\big( p_{u,v}^{w'} (X_u\otimes X_v) - s_i(p_{u,v}^{w'})\ (\de_{s_i} X_u\otimes\de_{s_i} X_v) \big)\\
&= \sum_{u,v}\ \tfrac{s_i(p_{u,v}^{w'})}{x_{\al_i}}\Big( -(\de_e-\de_{s_i})X_u\otimes(\de_e-\de_{s_i})X_v +(\de_e-\de_{s_i})X_u\otimes X_v\\
&\qquad\qquad + X_u\otimes (\de_e-\de_{s_i}) X_v  - X_u\otimes X_u\Big)+\tfrac{p_{u,v}^{w'}}{x_{\al_i}}(X_u\otimes X_u)\\
&= \sum_{u,v}\ -x_{\al_i}\cdot s_i(p_{u,v}^{w'})\ (X_iX_u\otimes X_iX_v) + s_i(p_{u,v}^{w'})\ (X_iX_u\otimes X_v + X_u\otimes X_iX_v)\\
&\qquad\qquad + X_i\circ p_{u,v}^{w'}\ (X_u\otimes X_v).
\end{align*}
Consider the product $X_iX_z$ in $\DF$.
Suppose that $s_i\notin D_L(z)$. Then $s_i z$ is reduced, and  $X_i X_{z}=X_{s_iz}$ in $\DF$. As in the cohomology case,  this corresponds to $s_i\in D_L(u)$, where $u=s_iz$.
Suppose that $s_i\in D_L(z)$. Then $z=s_iz'$ with $\ell(z')=\ell(z)-1$. So $X_iX_z=X_i^2 X_{z'}=tX_iX_{z'}=tX_z$ which corresponds to the case $s_i\in D_L(u)$, where $u=z$. The recursive formulas then follow.
\end{proof}

\begin{cor}
We have the following properties of the coefficients:
\begin{enumerate}
\item
If $p_{u,v}^w\neq 0$, then $u,v\leq w$ and
$p_{u,v}^w\in \II^k$, where $k=\ell(u)+\ell(v)-\ell(w)$.
\item
$\displaystyle p_{w,w}^w=(-1)^{\ell(w)}\prod_{\al\in \RS^+\cap w(\RS^-)} x_{\al}$ (i.e. the product of inversions of $w$).
\item
$p_{w,e}^w=p_{e,w}^w=1$.
\end{enumerate}
\end{cor}

\begin{ex} Similarly to Example~\ref{ex:rk2co} we then obtain for $w=s_2s_1$:
\begin{align*}
\DI(X_w) = x_{\al_2}s_1(x_{\al_1})X(w) &-s_2(x_{\al_1})X(w,s_1)-x_{\al_2}X(w,s_2)\\
&+X(w,e)-X_2\circ(x_{\al_1})X(s_1)+X(s_1,s_2).
\end{align*}
\end{ex}

Define the push-pull elements in $\mathcal{Q}$ via
\[
Y_i=\tfrac{1}{x_{-\al_i}}\de_e +\tfrac{1}{x_{\al_i}}\de_{s_i},\quad i=1\ldots n.
\]
Observe that for each $i$ we have
\[
X_i+Y_i=(\tfrac{1}{x_{\al_i}}+\tfrac{1}{x_{-\al_i}})\de_e =t\de_e.
\]
Then the formula for the coproduct in terms of $Y_i$s (cf Theorem~\ref{thm:coproduct_recursion}) will be:
\begin{align*}
\DI(Y_i) &=\DI(t -X_i)=t  - X_i\otimes 1-1\otimes X_i +x_{\al_i} X_i\otimes X_i\\
&= t   - (t -Y_i)\otimes 1 -1\otimes (t-Y_i)+x_{\al_i}(t-Y_i)\otimes (t-Y_i)\\
&= t^2x_{\al_i} +(1-tx_{\al_i})(Y_i\otimes 1 +1\otimes Y_i) +x_{\al_i} (Y_i\otimes Y_i).
\end{align*}
(here we identify $\de_e$ and $\de_e\otimes\de_e$ with $1$).

\subsection{\it The localization and the dual connective Hecke rings}
As in the cohomology case one may relate the dual of the connective Hecke ring to the equivariant connective $K$-theory of a flag variety.
We only briefly recall main results in this direction following \cite{CZZ,CZZ1,CZZ2} and \cite{KK90}.

One one side, by \cite[Theorem~10.7]{CZZ1} the natural inclusion $\SP_W \hra \DF$ induces a ring inclusion $\DF^* \hra \SP_W^*$ on the $\SP$-linear duals (called the moment map). On the other side by \cite[Theorem~8.2]{CZZ2} (Theorem~8.2 loc.cit.) the $T$-equivariant connective $K$-theory $\CK_T(G/B)$ can be identified with $\DF^*$ in such a way that the moment map coincides with the pull-back
\[
\imath^*\colon \DF^*=\CK_T(G/B) \hra \CK_T((G/B)^T)=\SP_W^*
\]
induced by the inclusion $\imath \colon (G/B)^T \hra G/B$.
Moreover, if $\II$ denotes the kernel of the augmentation map $\veps\colon \SP \to R$, $x_\la\mapsto 0$,
then there is the induced quotient map (here $\DF^*$ is an $\SP$-module)
\[
\veps^*\colon \DF^* \to \DF^*/\II\DF^*=\veps \DF^*,
\]
where the target can be identified with the usual connective $K$-theory $\CK(G/B)$ and the map $\veps^*$ coincides with the forgetful map
\[
\eta\colon \CK_T(G/B)\simeq \DF^*\to \veps \DF^* \simeq \CK(G/B).
\]
All these results can be expended to the parabolic situation in which case we obtain the forgetful map
$\CK_T(G/P_J) \to \CK(G/P_J)$.

By \cite[Prop.~7.7]{CZZ} the set $\{Y_{w}\}_{w}$ forms an $\SP$-basis of the algebra $\DF$.
We introduce coefficients $c_{w,v}$ via
\[
Y_w=\sum_v c_{w,v} \de_v,\quad c_{w,v}\in \mathcal{Q}.
\]
Let $\{\psi_w\}_{w\in W}$ and  $\{\xi_w\}_{w\in W}$ denote the Kronecker dual bases to
$\{\de_w\}_{w\in W}$ and $\{Y_w\}_{w\in W}$ respectively. We introduce the coefficients $d_{w,v}\in \mathcal{Q}$ as
\[
\xi_w=\sum_{v\in W}d_{w,v} \psi_v, \quad d_{w,v}.
\]
We then obtain the transformation matrices $C=[c_{w,v}]$, $D=[d_{w,v}]$ such that $C^{-1}=D^\mathrm{t}$.

As in the cohomology case, for each Schubert variety, we define the $T$-equivariant Schubert class
\[
\sigma_w:=[\overline{BwB/B}]\in \CK_T^*(G/B).
\]
The Schubert classes $\{\sigma_w\}_{w\in W}$ form a basis of $\DF^*=\CK_T^*(G/B)$ as a $\SP$-module, and
its evaluations $\{\veps^*\sigma_w\}_{w\in W}$ form a basis of $\CK(G/B)$ (see e.g. \cite[\S11]{CZZ}).
The main result of \cite{CZZ1} says that the bases $\{\xi_w\}_w$ and $\{\sigma_w\}_w$ are Poincar\'e-dual to each other.
Observe that
in the connective case the basis $\{\xi_w\}$ does not coincide with the Schubert basis. Indeed, when restricted to the usual $K$-theory it gives the so called ideal-sheaf basis.

\section{Calculations in SageMath}\label{sec:sagemath}

In this section, we present code written in SageMath (Python) for computing Schubert structure constants $p_{u,v}^w$ as defined in Theorem~\ref{thm:coproduct_recursion} for crystallographic root systems of finite type.  These coefficients are also the Schubert structure constants of the $T$-equivariant cohomology ring of the flag variety $G/B$ given in Theorem~\ref{thm:equiv_cohomology}.  The implementation of this code uses a certain coarsening of the formula given in~\eqref{eqn:BS_coeff} in combination with the recursion given in Theorem~\ref{thm:coproduct_recursion}.

Let $w,u,v\in W$ such that $m=\ell(w)$ and let $\ii, \jj, \kkk$ denote sequences corresponding to reduced words of $w,u,v$ respectively.  Define the coefficient
\[p_{\jj,\kkk}^{\ii}:=\sum_{E_1,E_2\subseteq[m]} p_{E_1,E_2}^\ii\]
where $p_{E_1,E_2}^\ii$ is defined in~\eqref{eqn:BS_coeff} and the sum is taken over all pairs $(E_1,E_2)\subseteq [m]^2$ such that $\ii_{E_1}=\jj$ and $\ii_{E_2}=\kkk$.  The next corollary follows from formulas~\eqref{Eqn:BS_coeff_operator} and~\eqref{eqn:BS_coeff}:

\begin{cor}
For any $w,u,v\in W$, and let $\ii$ be a sequence corresponding to a reduced word of $w$.  Then

\begin{equation}\label{eqn:coeff_SageMath}
p_{u,v}^w=\sum p_{\jj,\kkk}^{\ii}
\end{equation}
where the sum is over all sequences $\jj,\kkk$ corresponding to reduced words of $u,v$ respectively.

\smallskip

Furthermore, let $\ii',\jj',\kkk'$ denote subsequences of $\ii,\jj,\kkk$ obtained by removing the first entries $i_1, j_1$ and $k_1$ respectively.  Then $p_{\emptyset,\emptyset}^{\emptyset}=1$ and
\[
p_{\jj,\kkk}^\ii=\begin{cases}
Y_{i_1}\circ(p_{\jj,\kkk}^{\ii'})&\text{if $i_1\neq j_1$ and $i_1\neq k_1$}\\
Y_{i_1}\circ(p_{\jj,\kkk}^{\ii'})+s_{i_1}(p_{\jj',\kkk}^{\ii'})&\text{if $i_1=j_1$ and $i_1\neq k_1$}\\
Y_{i_1}\circ(p_{\jj,\kkk}^{\ii'})+s_{i_1}(p_{\jj,\kkk'}^{\ii'})&\text{if $i_1\neq j_1$ and $i_1= k_1$}\\
Y_{i_1}\circ(p_{\jj,\kkk}^{\ii'})+s_{i_1}(p_{\jj',\kkk}^{\ii'}+p_{\jj,\kkk'}^{\ii'})+\al_{i_1}s_{i_1}(p_{\jj',\kkk'}^{\ii'})&\text{if $i_1=j_1=k_1$.}
\end{cases}
\]
\end{cor}

In the proceeding SageMath code, the function \lstinline{Schub(w,u,v)} computes the coefficient $p_{\jj,\kkk}^\ii$ where $\ii,\jj,\kkk$ are fixed reduced words of $w,u,v$ respectively.  The function \lstinline{Schub_complete(w,u,v)} computes the coefficient $p_{u,v}^w$ by taking the sum of all $p_{\jj,\kkk}^\ii$ where $\jj, \kkk$ range over all reduced words of $u,v$ as in formula~\eqref{eqn:coeff_SageMath}.  This code was written using SageMath 9.6.  This version includes the class \lstinline{sage.combinat.root_system.root_system.RootSystem} which is required for this calculation.

\bigskip

\begin{lstlisting}[language=python,backgroundcolor=\color{yellow!10},frame=tlb]
Cartan=['A',8]  #Here you can input any finite Lie type and rank.
R=RootSystem(Cartan)
space = R.root_lattice()
alpha = space.simple_roots()
s = space.simple_reflections()
n = len(space.basis())
Sym=PolynomialRing(QQ, 'a', n)
W = WeylGroup(Cartan, prefix="s")

#Note that roots are labeled alpha[1], alpha[2], alpha[3]..., but the corresponding root variables in Sym are labled a0, a1, a2...

#notation for variables of Sym.
def al(k):
    return(Sym.gens()[k-1])

#converts a linear combination of roots to a linear combination of variables in Sym.
def root2var(X):
    Y=Sym(0) #this is zero in Sym
    for i in range(n):
        Y=Y+X.coefficient(i+1)*al(i+1)
    return(Y)

#converts a linear combination of variables in Sym to a linear combination of roots.
def var2root(Y):
    X=space(0) #this is zero in the root lattice
    for i in range(n):
        X=X+int(Y.coefficient(al(i+1))*alpha[i+1]
    return(X)

#returns the action of the simple reflection s_i on a tuple of roots.
def simple_action_tuple(i,tup):
    List=[]
    for T in tup:
        List.append(root2var(s[i](var2root(T))))
    return(tuple(List))

#returns the polynomial s_i(f) where s_i is a simple reflection.
def simple(f,i):
    if f==0:
        return(0)
    else:
        return(f(simple_action_tuple(i,Sym.gens())))

#returns the polynomial Y_i(f) where Y_i=(s_i-1)/alpha_i (Demazure operator).
def Dem(f,i):
    if f==0:
        return(0)
    else:
        return((simple(f,i)-f)/al(i))

#returns the Schubert coefficient p_{u,v}^w where w,u,v are fixed reduced words.
def Schub(w,u,v):
    if ((w!=u and v==[]) or (w!=v and u==[]) or (len(w)<len(u)) or (len(w)<len(v))):
        return(0)
    if ((w==u and v==[]) or (w==v and u==[])):
        return(Sym(1)) #this is 'one' in Sym
    k=w[0]
    if (k!=v[0] and k!=u[0]):
        return(Dem(Schub(w[1:],u,v),k))
    if (k==u[0] and k!=v[0]):
        return(Dem(Schub(w[1:],u,v),k)+simple(Schub(w[1:],u[1:],v),k))
    if (k==v[0] and k!=u[0]):
        return(Dem(Schub(w[1:],u,v),k)+simple(Schub(w[1:],u,v[1:]),k))
    if (k==v[0] and k==u[0]):
        return(Dem(Schub(w[1:],u,v),k)+simple(Schub(w[1:],u,v[1:])+Schub(w[1:],u[1:],v),k)+al(k)*simple(Schub(w[1:],u[1:],v[1:]),k))

#returns the complete Schubert coefficient p_{u,v}^w.  Here, w is a fixed reduced word and we take sum of Schub(w,u,v) where the sum is over all reduced words of u,v.
def Schub_complete(w,u,v):
    ulist=(W.from_reduced_word(u)).reduced_words()
    vlist=(W.from_reduced_word(v)).reduced_words()
    Total=Sym(0) #this is 'zero' in Sym
    for x in ulist:
        for y in vlist:
            Total=Total+Schub(w,x,y)
    return(Total)

#Examples:

Input: Schub_complete([1,5,4,3,2,3,4],[2,4],[4,3,2,3,4])
Output: 2

Input: Schub_complete([2,4,3,2], [2,3,2],[4,3,2])
Output: a1*a2 + a2^2 + a1*a3 + 2*a2*a3 + a3^2
\end{lstlisting}

\bibliographystyle{alpha}

\end{document}